\documentclass[reqno,11pt]{amsart}

\usepackage{amsmath, amssymb, amsthm, graphicx}

\newtheorem{lemma}{Lemma}[section]
\newtheorem{theorem}{Theorem}[section]
\newtheorem{proposition}{Proposition}[section]

\theoremstyle{definition}
\newtheorem{definition}{Definition}[section]
\theoremstyle{remark}
\newtheorem{remark}{Remark}[section]

\numberwithin{equation}{section}

\newcommand{\p}{\partial}
\newcommand{\norm}[1]{\left\Vert#1\right\Vert}
\newcommand{\dd}{\mathrm{d}}

\newcommand{\sign}{\mathrm{sign}}
\newcommand{\di}{\mathrm{div}}
\newcommand{\grad}{\mathrm{grad}}
\newcommand{\lie}{\mathcal{L}}
\newcommand{\vol}{\mathrm{vol}}

\makeatletter

\newcommand{\Rmnum}[1]{\mathrm{\expandafter\@slowromancap\romannumeral#1@}}
\makeatother
\begin{document}
\title[Local Uniqueness of Transonic Shocks]{Local Uniqueness of Steady Spherical
Transonic Shock-fronts for the Three--Dimensional Full Euler
Equations}
\author{Gui-Qiang G. Chen}
\author{Hairong Yuan}
\address{Gui-Qiang G. Chen: School of Mathematical Sciences, Fudan University,
 Shanghai 200433, China; Mathematical Institute, University of Oxford,
 Oxford, OX1, 3LB, UK;
Department of Mathematics, Northwestern University,
 Evanston, IL 60208-2730, USA}
\email{chengq@maths.ox.ac.uk}
\address{Hairong Yuan:
Department of Mathematics, East China Normal University, Shanghai
200241, China; Mathematical Institute, University of Oxford,
 Oxford, OX1, 3LB, UK}
\email{hryuan@math.ecnu.edu.cn,\ \
hairongyuan0110@gmail.com}

\date{\today}

\begin{abstract}
We establish the local uniqueness of steady transonic shock solutions
with spherical symmetry for the three-dimensional full Euler
equations. These transonic shock-fronts are important for
understanding transonic shock phenomena in divergent nozzles. From
mathematical point of view, we show the uniqueness of solutions of a free
boundary problem for a multidimensional quasilinear system of
mixed-composite elliptic--hyperbolic type. To this end, we develop a
decomposition of the Euler system which works in a general
Riemannian manifold, a method to study a Venttsel problem of
nonclassical nonlocal elliptic operators, and an iteration mapping
which possesses locally a unique fixed point. The approach reveals
an intrinsic structure of the steady Euler system and subtle
interactions of its elliptic and hyperbolic part.
\end{abstract}
\maketitle

%%%----------------------------------------------------------------------------

\section{Introduction}

The study of the Euler equations for compressible fluids is one of
the central topics in the mathematical fluid dynamics, and the
analysis of solutions to the system is of particular interest in
applications.  In particular, in the recent years, important
progress has been made in the analysis of transonic shock solutions
of the steady potential flow equation and the steady Euler system in
multidimensions (cf. \cite{CDSD,ChF1,ChF2,ChF3,CCF,CY,LY,Yu2,Yu3} and the
references cited therein).
% since the paper of Chen-Feldman \cite{ChF1}.

In this paper, we are concerned with the local uniqueness of
transonic shock solutions with spherical symmetry for the
three-dimensional, steady, full Euler system of polytropic gases.
Such a study not only helps us to understand transonic shock phenomena
occurred in divergent nozzles, which have many important applications,
but also provides new insights for the theory of free boundary problems
of partial differential equations of composite--mixed
elliptic--hyperbolic type.
This problem can be formulated as a free boundary problem for the
Euler system in a spherical shell, with the transonic shock-front as
a free boundary, which is a graph of a function defined on
$\mathbf{S}^2$ (the unit $2$--sphere in $\mathbb{R}^3$). Therefore,
for such a problem, although there exists a system of global
Descartes coordinates, it is more convenient to use the local
spherical coordinates and the terminology of differential geometry
(see Appendix A).

Let $r^0<r^1$ be two positive constants. A spherical shell in
$\mathbb{R}^3$ centered at the origin is a Riemannian manifold
$\mathcal{M}=[r^0, r^1]\times\mathbf{S^2}$ with a metric
$G=G_{ij}dx^i\otimes
dx^j$\renewcommand{\thefootnote}{\fnsymbol{footnote}}
\footnote[2]{In this paper we always use the Einstein summation
convention for the Roman indices from $0$ to $2$ and for the Greek
indices from $1$ to $2$.}
\renewcommand{\thefootnote}{\arabic{footnote}} in local spherical
coordinates (see Appendix A). Its boundary $\p\mathcal{M}$ is
$\Sigma^0\cup\Sigma^1$ with
$$
\Sigma^i:=\{(x^0,x)\in\mathcal{M}:
x^0=r^i, x=(x^1,x^2)\in\mathbf{S}^2\}, \qquad i=0,1,
$$
denoting respectively the entry and exit of $\mathcal{M}$.

Let $p, \rho$, and $S$ in $\mathcal{M}$ represent the pressure,
density, and entropy of gas flow in the manifold respectively. For
polytropic gases, $p=A(S)\rho^\gamma$ with the adiabatic exponent $\gamma>1$,
and the sonic speed is
$$
c=\sqrt{\gamma p/\rho}.
$$
Let $u$ be the velocity of fluid flow, which is a vector field in
$\mathcal{M}$ whose integral curves are called fluid trajectories in
$\mathcal{M}$.
Then the steady, full Euler equations for compressible fluids in
$\mathcal{M}$ are (cf. \S 16.5 in \cite{T}):
\begin{eqnarray}
&&\varphi:=\di (\rho u\otimes u)+\grad\, p=0,\label{101}\\
&&\varphi_1:=\di (\rho u)=0,\label{102}\\
&&\varphi_2:=\di (\rho E u)=0,\label{103}
\end{eqnarray}
where  $\di$ and $\grad$ are respectively the divergence and
gradient operator in $\mathcal{M}$, and
$$
E:=\frac{1}{2}G(u,u)+\frac{c^2}{\gamma-1}.
$$
We use $U=(p,\rho,u)$
to denote the state of fluid flow. If $U$ depends only on $x^0$ and
$u=u^0(x^0)\p_0$, then \eqref{101}--\eqref{103} can be reduced to
the following differential equations:
\begin{eqnarray}
&&\frac{\dd u^0}{\dd x^0}=\frac{2c^2u^0}{x^0((u^0)^2-c^2)},\label{1041}\\
&&\frac{\dd \rho}{\dd x^0}=-\frac{2\rho(u^0)^2}{x^0((u^0)^2-c^2)},\label{1042-a}\\
&&\frac{\dd p}{\dd x^0}=-\frac{2\rho
c^2(u^0)^2}{x^0((u^0)^2-c^2)}.\label{1042}
\end{eqnarray}

It has been shown in Yuan \cite{Yu3} that, for equations
\eqref{1041}--\eqref{1042},
given supersonic data
$U_b^-(r^0)=(p_b^-(r^0),\rho_b^-(r^0),u_b^-(r^0))$ on the entry
$\Sigma^0$, there exists an interval $I$ such that, if the back
pressure
$p_b^+(r^1)\in I$, then there exists a unique $r_b\in(r^0,r^1)$ so
that
$$
S_b=\{(x^0,x^1,x^2)\in\mathcal{M}\,:\, x^0=r_b,
(x^1,x^2)\in\mathbf{S}^2\}
$$
is a transonic shock-front, determined
by the Rankine-Hugoniot jump conditions: The flow $U_b^-(x^0)$,
$x^0\in(r^0,r_b),$ ahead of $S_b$ is supersonic; $U_b^+(x^0),
x^0\in(r_b,r^1),$ behind of $S_b$ is subsonic; and the physical
entropy condition $p_b^+(r_b)>p_b^-(r_b)$ holds on $S_b$. We call
such a spherical transonic shock solution $U_b:=(U_b^-, U_b^+; S_b)$
to the Euler equations as a {\em background solution}. The objective
of this paper is to study the uniqueness of transonic shock
solutions, if it exists, in a neighborhood of a class of background
solutions under the three-dimensional perturbations of the upcoming
supersonic flow $U_b^-$ at the entry $\Sigma^0$. This class of
background solutions especially include those background solutions
satisfying the {\it S-Condition} defined in \S 2. In particular, we
prove that, for a given back pressure, for sufficiently large
upstream pressure and Mach number, the background transonic
shock-front itself is locally unique. 

The main theorem of this paper is the following:

\begin{theorem}[Main Theorem]\label{thm101}
Let $U_b$ and $\gamma$ satisfy the S-condition (see Definition {\rm
2.3} in \S {\rm 2}). Then there exist $\varepsilon_0$ and $C_0$ depending
only on $U_b$ and $\gamma$ such that, if the upcoming supersonic
flow $U^-$ on $\Sigma^0$ satisfies
\begin{eqnarray}\label{th101}
\norm{U^--U_b^-}_{C^{3,\alpha}(\Sigma^0)}<\varepsilon\le\varepsilon_0
\end{eqnarray}
for some $\alpha\in(0,1)$ and there exists a  transonic shock
solution $U=(U^-,U^+; S^\psi)$ of \eqref{101}--\eqref{103} in
$\mathcal{M}$ satisfying Property {\rm (A)} below,
then this solution is unique.
Here Property {\rm (A)} consists of the three conditions:
\begin{itemize}
\item[(i)] $S^\psi=\{(x^0,x^1,x^2)\in\mathcal{M}\,:\, x^0=\psi(x^1,x^2),
(x^1,x^2)\in\mathbf{S}^2\}$ is the shock-front; $U^-$ is the
supersonic flow ahead of $S^\psi$; and $U^+$ is the subsonic flow
behind of $S^\psi$. The physical entropy condition:
$p^+|_{S^\psi}>p^-|_{S^\psi}$ holds on $S^\psi$.
Moreover, $\psi\in
C^{4,\alpha}(\mathbf{S}^2)$ satisfies
\begin{eqnarray}\label{th102}
\norm{\psi-r_b}_{C^{4,\alpha}(\mathbf{S}^2)}\le C_0\varepsilon;
\end{eqnarray}
\item[(ii)] The back pressure is
\begin{eqnarray} \label{th103}p(r^1,x^1,x^2)=p_b^+(r^1);
\end{eqnarray}
\item[(iii)] For
$\mathcal{M}_\psi^\pm=\{(x^0,x^1,x^2)\in\mathcal{M}\,:\, \pm (x^0-\psi(x^1,x^2))\ge 0,
 (x^1,x^2)\in\mathbf{S}^2\}$,
 \begin{eqnarray}\label{th104}
\norm{U^\pm-U^\pm_b}_{C^{3,\alpha}(\mathcal{M}_\psi^\pm)}\le
C_0\varepsilon,
\end{eqnarray}
\end{itemize}
with
$$
\norm{U}_{C^{k,\alpha}(\mathcal{M})}:=\norm{(p,
\rho)}_{C^{k,\alpha}(\mathcal{M})}
+\norm{\bar{u}}_{C^{k,\alpha}_1(\mathcal{M})},
$$
where $\bar{u}$ is the $1$-form corresponding to the vector field
$u$ via the Riemannian metric of $\mathcal{M}$, and $C_r^{k,\alpha}$
denotes the space of $r$-forms in $\mathcal{M}$ (i.e.,
${A}^r(\mathcal{M})$) with $C^{k,\alpha}$ components in local
coordinates and the norms are defined in the usual way by partition
of unity in the manifold.
\end{theorem}

As explained in \cite{CF,Yu3}, the background solutions coincide
with the solutions of the steady quasi-one-dimensional model of flows in
divergent nozzles. Therefore, the above uniqueness result will help
to understand transonic shock phenomena in divergent nozzles, as
well as the effectiveness of the quasi-one-dimensional model
\cite{W}. Also see \cite{Yu1} for an explanation of the
quasi-one-dimensional model from the viewpoint of flows in
Riemannian manifolds and \cite{LY} for the stability result of
transonic shock-fronts for the two-dimensional case.

Apart from the physical implications, the approach by considering
the Euler equations in a Riemannian manifold is of interest itself
in mathematics. We note that some studies have been made for
conservation laws in general Riemannian manifolds (cf. \cite{ABL,ALN,Si}
and the references cited therein). This approach via differential
geometry reveals some intrinsic structures of the steady Euler
system, which are valid in general Riemannian manifolds.

Finally, we remark that the existence and uniqueness of supersonic
flow $U^-$ in $\mathcal{M}$ subject to the initial data
$U^-|_{\Sigma^0}$ satisfying \eqref{th101} follow directly from the
theory of semi-global classical solutions of the Cauchy problem of
quasilinear symmetric hyperbolic systems if $\varepsilon_0$ is
sufficiently small (cf. \cite{Li,T}). Furthermore, one can obtain
\begin{eqnarray}\label{015}
\norm{U^--U^-_b}_{C^{3,\alpha}(\mathcal{M})}\le C_1\varepsilon,
\end{eqnarray}
where $C_1>0$ and $\varepsilon_0>0$  depend solely on $U_b^-(r^0)$
and $r^1$. Thus, we focus on showing the uniqueness of $\psi$ and $U^+$
below. Indeed, what we obtain here is much more than this: We design an
iteration mapping and show that it always has a unique fixed point, and
any solution to  \eqref{101}--\eqref{103} must be a fixed point of
this iteration mapping. Therefore, the solution to
\eqref{101}--\eqref{103} is unique, and then Theorem \ref{thm101} is
proved. However, we have not known whether the fixed point of the
iteration mapping is a solution to the original problem \eqref{101}--\eqref{103},
therefore, the existence problem for solutions to \eqref{101}--\eqref{103} is
still open, though we believe that the ideas and approaches developed here will
be useful to establish an existence theorem which is out of the scope of
this paper.

For simplicity, we write $U^+$ as $U$ from now on. We emphasize here
that the supersonic flow $U^-$ is defined in the whole
$\mathcal{M}$, while, by Proposition 9 in \cite{Yu3}, $U_b^+$ can be
extended to $[r_b-h_b, r^1]\times\mathbf{S}^2$ and still obeys the
Euler system, with $h_b>0$ depending only on $U_b^-(r^0).$

\medskip
The rest of this paper is organized as follows. In \S 2, we derive
some elliptic or transport equations, as well as certain equations
of exterior differential forms in $\mathcal{M}$ and on the
shock-front, from the Euler equations \eqref{101}--\eqref{103} and
the Rankine--Hugoniot jump conditions. Most of the formulas obtained
here are also valid in general Riemannian manifolds. Based on these
decompositions, in \S 3, we present an iteration mapping and
establish the existence of a unique fixed point, which yields Theorem
\ref{thm101}.
Some facts and notations of
differential geometry
are shown and described in Appendix A.

We remark in passing that the analysis and results developed here should be
straightforward
extended to the higher dimensional case,
even general Riemannian manifolds for most of them.

\section{Reduction of the Euler system and Rankine-Hugoniot Jump Conditions}

In this section we introduce a reduction of the Euler system and analyze
the Rankine-Hugoniot jump conditions.

\subsection{The Euler Equations in $\mathcal{M}$}\label{sec102}

We use $d$ to denote the exterior differential operator and $D_u
\omega$ to denote the covariant derivative of  a tensor field
$\omega$ with respect to a vector field $u$ in $\mathcal{M}$; while
$\nabla_u \omega$ is the covariant derivative on $\mathbf{S}^2$.
$\mathcal{L}_u \omega$ is the Lie derivative of $\omega$ with
respect to $u$ in $\mathcal{M}$. The symbol $\Delta$ represents the
Laplacian of forms in $\mathcal{M}$, and $\Delta'$ is the Laplacian
of forms on $\mathbf{S}^2$, which are both positive operators (cf.
\eqref{A.11}).
Note that $\psi\in
C^{4,\alpha}(\mathbf{S}^2)$ defines a mapping from $\mathbf{S}^2$ to
$\mathcal{M}$ by $(x^1,x^2)\mapsto(\psi(x^1,x^2),x^1,x^2)$. We use
$\psi^*$ to denote the pull back of forms and functions induced by
this mapping; for example, for a function $p\in A^0(\mathcal{M})$,
$\psi^*p=p|_{S^\psi}$.  The volume $2$-form of $\mathbf{S}^2$ is
written as $\vol$, and $\vol^3$ is the volume $3$-form of
$\mathcal{M}$.

In the following, we derive some well-known equations from the Euler
system \eqref{101}--\eqref{103} which are valid only for $C^1$ flows
(cf. \cite{T}). Since these involve differentiations of
\eqref{101}--\eqref{103}, the solution of the reduced equations
might not be a solution to the Euler system
\eqref{101}--\eqref{103}. One point below is to express the
relations between these derived equations and the original Euler
equations, which may be useful in the future to verify that the
solution of these derived equations satisfies the Euler system
indeed.

The conservation of mass \eqref{102} can be written (equivalently
for $C^1$ flows) as
\begin{eqnarray}
\varphi_1=D_u \rho+\rho\, \di\, u=0.
\end{eqnarray}
By the identity
\[
\di (u\otimes v)=(\di\, v)u+ D_v u
\]
for two vector fields $u$ and $v$, the momentum equations
\eqref{101} become
\begin{eqnarray}\label{202}
\varphi_0:=\frac{1}{\rho}({\varphi-\varphi_1 u})= D_u
u+\frac{1}{\rho}\grad\, p=0.
\end{eqnarray}
Similarly, the conservation law of energy \eqref{103} may be written as
\begin{eqnarray}\label{106}
\varphi_3:=\frac{1}{\rho}({\varphi_2-E\varphi_1})=D_u E=0.
\end{eqnarray}
This is exactly the well-known Bernoulli law.

Since
\begin{eqnarray*}
\frac 12 D_u\big(|u|^2\big)=G(D_u u, u)=G(\varphi_0-\frac{1}{\rho}
\grad\, p,\, u)=G(\varphi_0, u)-\frac 1\rho D_up,
\end{eqnarray*}
we have
\begin{eqnarray}\label{107}
\varphi_4&:=&\frac{1}{\rho^{\gamma-1}}\big(\varphi_3-G(\varphi_0,
u)\big)=\frac{1}{\rho^{\gamma-1}}\Big(\frac{\gamma}{\gamma-1}D_u\big(A(S)\rho^{\gamma-1}\big)
 -\frac{1}{\rho}D_u(A(S)\rho^\gamma)\Big)\nonumber\\
&=&D_uA(S)=0,
\end{eqnarray}
i.e., the invariance of entropy along the flow trajectories for
$C^1$ flows.

For a vector field $u=u^i\p_i$ in $\mathcal{M}$, we always use
$\bar{u}=u^jG_{ij}dx^i$ to denote its corresponding $1$-form with
respect to the metric $G$ of $\mathcal{M}$. Then \eqref{202} is
equivalent to
\begin{eqnarray}
\bar{\varphi}_0=D_u \bar{u}+
\frac{dp}{\rho}=\mathcal{L}_u\bar{u}-d(\frac{|u|^2}{2})+\frac{dp}{\rho}=0,
\label{108}
\end{eqnarray}
since $ \mathcal{L}_u \bar{u}=D_u\bar{u}+d(\frac{|u|^2}{2})$. This
implies
\begin{equation}\label{183}
\mathcal{L}_ud\bar{u}=-d\big(\frac{1}{\rho}\big)\wedge
dp+d\bar{\varphi}_0,
\end{equation}
which is a transport equation of vorticity. Moreover,
$d\bar{u}|_{S^\psi}$, the initial value of vorticity on $S^\psi$,
expressed in local spherical coordinates, is
\begin{eqnarray}\label{185}
d\bar{u}|_{S^\psi}&=&d\big(\psi^*(u^iG_{ij})\big)\wedge
dx^j-\big(\psi^*\p_0(u^iG_{ij})\big)d\psi\wedge dx^j+\big(\psi
g_{\alpha\beta}\psi^*u^\beta-\psi^*(\frac{\p_\alpha p}{\rho
u^0})\big)dx^0\wedge dx^{\alpha}\nonumber\\
&&\ +\psi^2
g_{\alpha\beta}\psi^*(\frac{dx^\alpha(\varphi_0)}{u^0})dx^0\wedge
dx^\beta,
\end{eqnarray}
where $g=g_{\alpha\beta}dx^\alpha\otimes dx^\beta$ is the standard
metric of $\mathbf{S}^2$.

Let $d^*$ be the codifferential operator in $\mathcal{M}$. Using the
identity $d^*\bar{u}=-\di\, u$, we obtain
\begin{eqnarray}
\varphi_1=-d^*(\rho\bar{u})=0.\label{110}
\end{eqnarray}
Let $\langle\cdot,\, \cdot\rangle$ be the inner product in
$\mathcal{M}$ of forms and let $*$ be the Hodge star operator, which
imply
$$
*\vol^3=1,\qquad
*1=\vol^3,\qquad d(*\bar{u})=(\di\, u)\,\vol^3, \qquad \alpha\wedge*\beta=\langle
\alpha, \beta\rangle\,\vol^3.
$$
Then \eqref{110} is
\begin{eqnarray}
{\varphi}_1=-\rho d^*\bar{u}+\langle d\rho,
\bar{u}\rangle.\label{111}
\end{eqnarray}
Note that, by the equation of state $p=A(S)\rho^\gamma$, we have
$$
\langle d\rho, \bar{u}\rangle=D_u\rho=\frac{D_u p}{c^2}-\frac{\p
p}{\p A}\, \frac{D_u A(S)}{c^2}.
$$
Thus,  by setting
\begin{eqnarray}\label{phi1}
\bar{\varphi}_1:=\frac{\varphi_1}{\rho}+\frac{1}{\rho
c^2}\big(\varphi_2-E\varphi_1-\rho G(\varphi_0,u)\big),
\end{eqnarray}
the equation of conservation of mass may be written as
\begin{eqnarray}
\bar{\varphi}_1=-d^*\bar{u}+\frac{D_u p}{\gamma p}=\di\, u+\frac{D_u
p}{\gamma p}=0.
\end{eqnarray}
By this equation and the definition of the Laplacian of forms
$\Delta=dd^*+d^*d$, we also have
\begin{eqnarray}\label{186}
\Delta {\bar{u}}=d^*(d\bar{u})+d(\frac{D_up}{\gamma
p})-d\bar{\varphi}_1.
\end{eqnarray}

\subsection{The Rankine--Hugoniot Jump Conditions on a Shock-Front}\label{sec103}

A shock-front $S^\psi$ is a hyper-surface in $\mathcal{M}$ across
which the physical variables have a jump. In our case, it can be
expressed as the graph of a mapping $\psi: \mathbf{S}^2\rightarrow
\mathcal{M}$. In local spherical coordinates, we may write
\begin{eqnarray}
S^\psi=\{(x^0,x^1,x^2)\in\mathcal{M}\,:\, x^0=\psi(x^1,x^2),\,
(x^1,x^2)\in\mathbf{S}^2\}.\label{1301}
\end{eqnarray}
The normal vector field and the corresponding normal $1$-form of
$S^\psi$ with respect to $\mathcal{M}$ are
\[
n=(\p_\alpha\psi G^{\alpha\beta}\p_\beta-G^{00}\p_0)|_{S^\psi}\qquad
\text{and}\qquad \bar{n}=(d\psi-dx^0)|_{S^\psi}.
\]

From \eqref{101}--\eqref{103}, the Rankine--Hugoniot jump conditions
(i.e., the R--H conditions) on $S^\psi$ are
\begin{eqnarray}
\lceil G(u,n)\rho u+pn\rfloor|_{S^\psi}=0
&\,\, \Longleftrightarrow\,\,
&\lceil \bar{n}(u)\rho \bar{u}+p\bar{n}\rfloor|_{S^\psi}=0,\label{115}\\
\lceil G(u,n)\rho \rfloor|_{S^\psi}=0&\Longleftrightarrow&
\lceil \bar{n}(u)\rho \rfloor|_{S^\psi}=0,\label{116}\\
\lceil G(u,n)\rho E \rfloor|_{S^\psi}=0&\Longleftrightarrow& \lceil
\bar{n}(u)\rho E \rfloor|_{S^\psi}=0,\label{117}
\end{eqnarray}
where $\lceil\cdot\rfloor$ denotes the jump of a quantity across
$S^\psi$. It is well-known (see \cite{CF,Da}) that a piecewise $C^1$
state $U=(U^-,U^+; S^\psi)$ is a weak entropy solution of
\eqref{101}--\eqref{103} if and only if $U$ satisfies these
equations in $\mathcal{M}_\psi^\pm$ in the classical sense,
the R--H conditions along $S^\psi$, as well as the physical
entropy condition $p^+|_{S^\psi}>p^-|_{S^\psi}$ on $S^\psi$.

Set
\begin{eqnarray}
&&u=u_0+u_1:=u^0\,\p_0+u^\alpha\,\p_\alpha,\\
&&\bar{u}=\bar{u}^0+\bar{u}^1:=u^0G_{00}\, dx^0+u^\alpha
G_{\alpha\beta}\, dx^\beta.\label{222}
\end{eqnarray}
Then $\bar{n}(u)=d\psi(u_1)-u^0$, and \eqref{115} can be separated
into
\begin{eqnarray}
&&\psi^*\lceil pd\psi-g\bar{u}^1\rfloor=0,\label{120}\\
&&\psi^*\lceil g u^0+p\rfloor=0,\label{121}
\end{eqnarray}
where $ g=g(U,\psi,D\psi):=(\rho u^0-d\psi(\rho u_1))|_{S^\psi}$ is
a function on $\mathbf{S}^2$,  with $d\psi\in A^1(\mathbf{S}^2)$
being considered as a $1$-form in $\mathcal{M}$; and the expression
$g(U,\psi,D\psi)$ means that $g$ depends on $U, \psi$, and the first-order
derivatives of $\psi$. Then \eqref{116}--\eqref{117} become
\begin{eqnarray} \psi^*\lceil
g\rfloor=0,\qquad \psi^*\lceil E\rfloor=0.\label{123}
\end{eqnarray}
Equations \eqref{106} and \eqref{123} indicate that $E$ is a
constant along the same trajectory even across the shock-front.
Therefore, we may write $E|_{S^\psi}=E_0(x), x\in \mathbf{S}^2$,
with $E_0(x)$ a given function depending only on the supersonic data
at the entry.

Now, if $\psi^*\lceil p\rfloor\ne 0$ (which is guaranteed later by
the physical entropy condition of the background solution),
\eqref{120} yields
\begin{eqnarray}
d\psi=\omega:=\frac{g\psi^*(\lceil \bar{u}^1\rfloor)}{\psi^*\lceil
p\rfloor}=\mu_0\psi^*(\bar{u}^1)+g_0(U,U^-,\psi,D\psi)\in
A^1(\mathbf{S}^2) \label{124}
\end{eqnarray}
by using the first equation in \eqref{123}, with
$$
\mu_0:=\left.\frac{(\rho u^0)_b^+}{p_b^+-p_b^-}\right|_{S_b}=\left.
 \frac{\gamma+1}{2}\frac{\big((u^{0})_b^+\big)^2}{\big(c^2-(u^{0})^2\big)_b^+}\right|_{S_b}
 >0,
$$
and $g_0$ the higher order term defined below (see Definition
\ref{def102})\footnote{From now on, we always use
$g_i$ to denote the higher order terms on $S^\psi$, and $f_i$ to
denote the higher order terms in $\mathcal{M}_\psi^+.$}, which
contains $U,U^-,\psi$, and $D\psi$ (the first-order derivatives of
$\psi$). We note that
equations \eqref{121}--\eqref{124} are equivalent to
\eqref{115}--\eqref{117}. Equation \eqref{124} also indicates
$d\omega=0$. Therefore, we have
\begin{eqnarray}\label{125}
d(\psi^*\bar{u}^1)=\chi(U,U^-,\psi):=-\frac{dg_0}{\mu_0}.
\end{eqnarray}

\begin{definition}\label{def101}
A constant $r^p$ is called the {\it position} of the surface
$S^\psi$ defined by \eqref{1301} provided that
$\int_{\mathbf{S}^2}(\psi-r^p)\,\vol=0.$ The function
$\psi^p:=\psi-r^p$ is called the {\it profile} of $S^\psi$.
\end{definition}

\begin{remark}
The reason why we distinguish the ``position" and the ``profile" is
that they are determined by different mechanisms: The ``profile" is
determined by the R--H conditions, while the ``position" is
determined by the solvability conditions closely related to the
conservation of mass.
\end{remark}

\begin{definition}\label{def102}
Let $\hat{U}=U^+-U_b^+.$ A higher order term is an expression
that contains either

(i) $U^--U_b^-$ and its first-order derivatives;

\noindent
or

(ii) the products of $\psi^p,\ r^p-r_b, \hat{U}$, and their
derivatives $D\hat{U},D^2\hat{U}, D\psi,D^2\psi$, and $D^3\psi$,
where $D^ku$ are the $k^{th}$-derivatives of $u$ in local
coordinates.
\end{definition}

Next, we linearize the R--H conditions \eqref{121}--\eqref{123}. We
write them  equivalently as
$$
G_i(\psi^*U, \psi^*U^-)=\Psi_i(\psi^*U, \psi^*U^-,D\psi),\qquad
i=1,2,3,
$$
with
\begin{align}
& G_1=\psi^*\lceil \rho(u^0)^2+p\rfloor, & \Psi_1&=\psi^*(\lceil
d\psi(\rho u^0u_1)\rfloor),\\
& G_2=\psi^*\lceil \rho u^0\rfloor, & \Psi_2&=\psi^*(\lceil
d\psi(\rho u_1)\rfloor),\\
& G_3=\psi^*\lceil E\rfloor, & \Psi_3&=0.
\end{align}
As in \cite{LY}, since $G_i(U_b^+(r_b,x),U_b^-(r_b,x))=0$ for
$x=(x^1,x^2)\in\mathbf{S^2},$ we have
\begin{eqnarray}\label{1061}
&&\p_+G_i(U_b^+(r_b,x),U_b^-(r_b,x))\bullet\big(U(\psi(x),x)
-U_b^+(\psi(x),x)\big)\\
&&=\Big\{-\big(\p_+G_i(U_b^+(\psi(x),x),U_b^-(\psi(x),x))
-\p_+G_i(U_b^+(r_b,x),U_b^-(r_b,x))\big)\nonumber\\
&&\qquad\quad\,\,\,\bullet\big(U(\psi(x),x)-U_b^+(\psi(x),x)\big)\nonumber\\
&&\qquad+\p_+G_i(U_b^+(\psi(x),x),U_b^-(\psi(x),x))\bullet\big(U(\psi(x),
x)-U_b^+(\psi(x),x)\big)\nonumber\\
&&\qquad
-\big(G_i(U(\psi(x),x),U^-(\psi(x),x))-G_i(U_b^+(\psi(x),x),
U^-(\psi(x),x))\big)\nonumber\\
&&\qquad
-\big(G_i(U_b^+(\psi(x),x),U^-(\psi(x),x))-G_i(U_b^+(\psi(x),
x),U_b^-(\psi(x),x))\big) +\Psi_i\Big\}\nonumber\\
&&\quad-\Big\{G_i(U_b^+(\psi(x),x),U_b^-(\psi(x),x))-G_i(U_b^+(r_b,x),
U_b^-(r_b,x))\Big\}\nonumber\\
&&=:\Rmnum{1}_i + \Rmnum{2}_i,\nonumber
\end{eqnarray}
where we use ``$\bullet$" as the scalar product of vectors in the
phase (Euclidean) space, and $\p_+G_i(U,U^-)$ and $\p_-G_i(U,U^-)$ as the
gradient of $G_i(U,U^-)$ with respect to the variables $U$ and $U^-$,
respectively.

By the Taylor expansion, the terms in ${\Rmnum{1}_i}$
are of higher order. However,
\begin{eqnarray*}
&&\Rmnum{2}_1=\frac{2}{r_b}\big(p_b^+(r_b)-p_b^-(r_b)\big)(\psi^p
+r^p-r_b)+O(|\psi-r_b|^2), \qquad\,\,\Rmnum{2}_j=0,\ \ j=2,3.
\end{eqnarray*}
The Landau symbol $O(|\psi|^2)$ means the terms of order at least to
be two of $\psi$. One can also obtain
\begin{eqnarray*}
d_0
&:=&\det\Big(\left.\frac{\p(G_1,G_2,G_3)}
{\p(u^0,p,\rho)}\Big)\right|_{{(U_b^-,U_b^+; S_b)}}\nonumber
\\
&=&\det \left.\left(
  \begin{array}{ccc}
    2\rho_b^+ (u^0)_b^+ & 1 & ((u^0)_b^+)^2 \\
    \rho_b^+ & 0 & (u^0)_b^+ \\
    (u^0)_b^+ & \frac{\gamma}{\gamma-1}\frac{1}{\rho_b^+}
    & -\frac{(c^2)_b^+}{\gamma-1}\frac{1}{\rho_b^+} \\
  \end{array}
\right)\right|_{x^0=r_b}\nonumber\\
&=&\frac{(c^2-(u^0)^2)_b^+(r_b)}{\gamma-1}>0.
\end{eqnarray*}
Thus, \eqref{1061} is equal to
\begin{eqnarray*}
&&\left.\left(
  \begin{array}{ccc}
    2\rho_b^+ (u^0)_b^+ & 1 & ((u^0)_b^+)^2 \\
    \rho_b^+ & 0 & (u^0)_b^+ \\
    (u^0)_b^+ & \frac{\gamma}{\gamma-1}\frac{1}{\rho_b^+}
    & -\frac{(c^2)_b^+}{\gamma-1}\frac{1}{\rho_b^+} \\
  \end{array}
\right)\right|_{x^0=r_b}
\left(  \begin{array}{c}
     \psi^*(\widehat{u^0})\\
    \psi^*\hat{p} \\
    \psi^*\hat{\rho} \\
  \end{array}
\right)\nonumber\\
&&=
\left(
  \begin{array}{c}
     -\frac{2}{r_b}\big(p_b^+(r_b)-p_b^-(r_b)\big)\\
    0 \\
    0 \\
  \end{array}
\right)(\psi-r_b)+\text{h.o.t.},
\end{eqnarray*}
where {\it h.o.t.} represents  the {\it higher order terms} for
short.

We can solve this linear system to obtain
\begin{eqnarray}
&&\psi^*(\widehat{u^0})=\mu_1\,(\psi^p+r^p-r_b)+g_1(U,U^-,\psi,D\psi),
\label{143}\\
&&\psi^*(\hat{p})=\mu_2\,(\psi^p+r^p-r_b)+g_2(U,U^-,\psi,D\psi),
\label{144}\\
&&\psi^*(\hat{\rho})=\mu_3\,(\psi^p+r^p-r_b)+g_3(U,U^-,\psi,D\psi).
\label{145}
\end{eqnarray}
Using $A(S)=p\rho^{-\gamma}$, we also obtain
\begin{eqnarray}
\psi^*(\widehat{A(S)})=\mu_4\,(\psi^p+r^p-r_b)+g_4(U,U^-,\psi,D\psi),\label{146}
\end{eqnarray}
where
\begin{eqnarray*}
&&\mu_1=\frac{4\gamma(u^0)^+_b(r_b)}{(\gamma+1)r_b}>0,\qquad
\mu_2=-\frac{4\rho_b^+(r_b)}{(\gamma+1)r_b}\big((\gamma-1)(u^0)^2+c^2\big)_b^+(r_b)<0,\\
&&\mu_3=-\frac{4\gamma\rho_b^+(r_b)}{(\gamma+1)r_b}<0,\qquad
\mu_4=\frac{4(\gamma-1)}{(\gamma+1)r_b(\rho_b^+(r_b))^{\gamma-1}}
   (c^2-(u^0)^2)_b^+(r_b)>0.
\end{eqnarray*}

\subsection{Restriction of the Conservation of Mass and Momentum on the Shock-Front}
\label{sec107}

We now calculate $d^*(\psi^*\bar{u}_1)$ by restricting the equations
of conservation of mass and momentum on the shock-front $S^\psi$ and
obtain a second-order elliptic equation for $\psi^p$.

In local spherical coordinates, we have
\begin{eqnarray*}
d^*(\psi^*\bar{u}^1)&=&-\di(\psi^*(u^\alpha
G_{\alpha\beta})g^{\beta\gamma}\p_{\gamma})\nonumber\\
&=&-\frac{1}{\sqrt{g}}\p_\alpha(\sqrt{g}\psi^2\psi^*u^\alpha)\nonumber\\
&=&-\psi^2\psi^*\big(\frac{1}{\sqrt{G}}\p_\alpha(\sqrt{G}u^\alpha)\big)
-\psi^*\p_0\big((x^0)^2d\psi(u_1)\big)\nonumber\\
&=&-\psi^2\psi^*\big(\di\,
u-\frac{1}{\sqrt{G}}\p_0(\sqrt{G}u^0)\big)-\psi^*\p_0\big((x^0)^2d\psi(u_1)\big)\nonumber\\
&=&\psi^2\psi^*(d^*\bar{u})+\psi^2\psi^*\big(\frac{1}{\sqrt{G}}\p_0(\sqrt{G}u^0)\big)
-\psi^*\p_0\big((x^0)^2d\psi(u_1)\big)\nonumber\\
&=&\psi^2\psi^*\big(\frac{D_up}{\gamma
p}-\bar{\varphi}_1\big)+\psi^*\p_0\big((x^0)^2u^0\big)
-\psi^*\p_0\big((x^0)^2d\psi(u_1)\big).
\end{eqnarray*}
On the other hand, from the conservation of momentum,
\begin{eqnarray*}
\psi^*\big(dx^0(\varphi_0)\big)&=&\psi^*\big(u^0\p_0u^0\big)+\psi^*\big(\frac{\p_0
p}{\rho}\big)+\psi^*\big(u^\alpha\p_\alpha u^0+u^\alpha
u^\beta\Gamma_{\alpha\beta}^0\big)\nonumber\\
&=&\psi^*\big(u^0\p_0u^0\big)+\psi^*\big(\frac{\p_0p}{\rho}\big)
+\nabla_{\psi^*u_1}(\psi^*u^0)\nonumber\\
&&-d\psi(\psi^*u_1)\big(\psi^*\p_0u^0\big)-\psi\,
g(\psi^*u_1,\psi^*u_1),
\end{eqnarray*}
where we have set $\psi^*u_1:=(\psi^*u^\alpha)\p_\alpha$ which is a
well-defined vector field on $\mathbf{S}^2$,
$\nabla_{\psi^*u_1}(\psi^*u^0)$ is a higher order term, and
$\Gamma_{\alpha\beta}^0$ are the Christoffel symbols (see Section
A.1). Then we obtain
\begin{eqnarray*}
d^*(\psi^*\bar{u}^1)=\Rmnum{1}+\Rmnum{2}+\Rmnum{3},
\end{eqnarray*}
with
\begin{eqnarray*}
&&\Rmnum{1}=\psi^2\psi^*\Big(\frac{dx^0(\varphi_0)}{u^0}-\bar{\varphi}_1\Big)
=\psi^2\psi^*\Big(\frac{dx^0(\varphi_0)}{(u^0)_b^+}-\bar{\varphi}_1\Big)
+\text{h.o.t.},\\
&&\Rmnum{2}=\psi^2\psi^*\Big(\frac{(u^0)^2-c^2}{\gamma pu^0}\p_0
p\Big)+2\psi\psi^*(u^0),\\
&&\Rmnum{3}=\frac{\psi^2}{\psi^*u^0}\big(\psi\, |\psi^*u_1|^2
+d\psi(\psi^*u_1)\, (\psi^*\p_0u^0)-\nabla_{\psi^*u_1}(\psi^*u^0)\big)\nonumber\\
&&\qquad-\psi^*\p_0\big((x^0)^2d\psi(u_1)\big)+\frac{\psi^2}{\gamma
p}\big(\nabla_{\psi^*u_1}(\psi^*p)-d\psi(\psi^*u_1)(\psi^*\p_0p)\big),
\end{eqnarray*}
where, in the term $\Rmnum{1}$, we used that $\varphi_0=
\varphi_0-(\varphi_0)_b^+$ is small so that
$dx^0(\varphi_0)(\frac{1}{u^0}-\frac{1}{u^0_b})$ is a higher order
term.

Note that, for the background solution, $\Rmnum{2}=0$. Then the
Taylor expansion and the boundary conditions
\eqref{143}--\eqref{146} yield
\begin{eqnarray}\label{155}
\Rmnum{2}=\mu_5\, \psi^*(\p_0\hat{p})
 +\mu_6\, \psi^p +\mu_6\,(r^p-r_b)+O(|(\psi^p,r^p-r_b,
\psi^*(\hat{U}))|^2),
\end{eqnarray}
with
\begin{eqnarray*}
&&\mu_5=\left. \frac{r_b^2((u^0)^2-c^2)_b^+}{\gamma p_b^+ (u^0)_b^+}\right|_{r=r_b}<0,\\
&&\mu_6=\frac{8\gamma
(u^0)_b^+(r_b)}{(\gamma+1)(1-t_s)}\big((\gamma-1)t^2_s+t_s+1\big)>0,
\end{eqnarray*}
where $t_s=t(r_b)
=(M_b^+)^2(r_b)\in(0,1)$, and
$M_b^+(r_b)=\frac{(u^0)_b^+(r_b)}{c_b^+(r_b)}$ is the Mach number of
the flow behind the transonic shock-front of the background
solution.

We note that the term $\Rmnum{3}$  consists of higher order terms. Then we
obtain
\begin{eqnarray}\label{156} d^*(\psi^*\bar{u}^1)&=&\mu_5\,
\psi^*(\p_0\hat{p})+\mu_6\,
\psi^p+\mu_6\,(r^p-r_b)\nonumber\\
&&+g_5(U,U^-,\psi,DU,D\psi)+
\psi^2\psi^*\Big(\frac{dx^0(\varphi_0)}{(u^0)_b^+}-\bar{\varphi}_1\Big).
\end{eqnarray}
Therefore, by \eqref{124}, we have
\begin{eqnarray}\label{psieq}
\Delta'\psi^p+\mu_7\psi^p&=&\mu_0\mu_6(r^p-r_b)+\mu_0\mu_5\,\psi^*\p_0\hat{p}\nonumber\\
&&+g_6(U,U^-,\psi,DU,D\psi,D^2\psi)+\mu_0\psi^2\psi^*
\Big(\frac{dx^0(\varphi_0)}{(u^0)_b^+}-\bar{\varphi}_1\Big),
\end{eqnarray}
with $g_6=\mu_0g_5+d^*g_0$ and $\mu_7=-\mu_0\mu_6<0$.

By \eqref{156} and the divergence theorem, we choose
\begin{eqnarray}\label{req}
r^p-r_b=-\frac{1}{4\pi\mu_6}\int_{\mathbf{S}^2}\Big(\mu_5\,
\psi^*(\p_0\hat{p})+g_5(U,U^-,\psi,DU,D\psi)\Big)\, \vol.
\end{eqnarray}
Substituting this into \eqref{144}, we obtain
\begin{eqnarray}
\psi^p=\frac{1}{\mu_2}\left(\psi^*\hat{p}-{\mu_8}
\int_{\mathbf{S}^2}\psi^*(\p_0\hat{p})\,\vol+{g_7(U,U^-,\psi,D\psi)}\right),
\label{177}
\end{eqnarray}
with $\mu_8=-\frac{\mu_2\mu_5}{4\pi\mu_6}<0$ and
$g_7=\frac{\mu_2}{4\pi\mu_6}\int_{\mathbf{S}^2}g_5\,\vol-g_2.$
Therefore, from \eqref{psieq}, we also obtain an equation for the
pressure on $S^\psi$:
\begin{eqnarray}\label{178}
&&\Delta'(\psi^*\hat{p})+\mu_7(\psi^*{\hat{p}})+\mu_9(\psi^*\p_0\hat{p})\nonumber\\
&&=g_8(U,U^-,\psi,DU,D\psi,D^2U,D^2\psi,D^3\psi)
 +\mu_2\mu_0\psi^2\psi^*\big(\frac{dx^0(\varphi_0)}{(u^0)_b^+}-\bar{\varphi}_1\big),
\end{eqnarray}
where $\mu_9=-\mu_0\mu_2\mu_5<0$ and
$g_8=\Delta'g_2+\mu_7g_2+\mu_0\mu_2g_5+\mu_2d^*g_0.$

\subsection{An Elliptic Equation for the Pressure in $\mathcal{M}_\psi^+$}

First, we note the following tensor identity:
\begin{eqnarray*}
\di (D_uu)-D_u(\di u)=C_1^1C^1_2(Du\otimes Du)+\mathrm{Ric}(u,u),
\end{eqnarray*}
where $Du$ is the covariant differential of the vector field $u$ in
a Riemannian manifold, $\mathrm{Ric}(\cdot, \cdot)$ is the Ricci
curvature tensor, and $C^i_j(T)$ is the contraction on the upper $i$
and lower $j$ indices of a tensor $T$. In our case, since
$\mathcal{M}$ is flat, $\mathrm{Ric}(u,u)\equiv0.$ From the Euler
equations, we have
\begin{eqnarray*}
\di (D_uu)-D_u(\di\, u)=\di\,
\varphi_0-D_u\bar{\varphi}_1+D_u\big(\frac{D_u p}{\gamma p}\big)-\di
\big(\frac{\grad\, p}{\rho}\big),
\end{eqnarray*}
while direct calculation yields that, for $\p_0^2=\p_0\p_0,$
\begin{eqnarray*}
D_u\big(\frac{D_u p}{\gamma p}\big)-\di \big(\frac{\grad\,
p}{\rho}\big)&=&\frac{1}{\gamma
p}\Big(((u^0)^2-c^2)\p_0^2p-\frac{2c^2}{x^0}\p_0p
+\frac{c^2}{(x^0)^2}\Delta'p\Big)\nonumber\\
&&+\frac{1}{\gamma p^2}\Big((c^2-(u^0)^2)
(\p_0p)^2+p\p_0p\, \p_0\big(\frac{(u^0)^2}{2}-c^2\big)\Big)\nonumber\\
&&+f_1(U,DU,D^2U),\\
 C^1_1C^1_2(D u\otimes D
u)-\frac{2u^0}{x^0}\bar{\varphi}_1&=&(\p_0
u^0)^2-\frac{2}{(x^0)^2}(u^0)^2-\frac{1}{x^0}\p_0(u^0)^2-\frac{2}{x^0}
\,\frac{(u^0)^2}{\gamma p}\p_0 p\nonumber\\
&&+f_2(U,DU,D^2U).
\end{eqnarray*}
Now the point is that the above right-hand sides may be expressed as
functions of $\p_0^2p, \p_0p, p$, $\psi^*A(S)$, and some higher
order terms.

Indeed, due to the conservation of momentum,
\begin{eqnarray*}
\p_0u^0=\frac{1}{u^0}dx^0(\varphi_0)-\frac{1}{\rho
u^0}\p_0p+f_3(U,DU),
\end{eqnarray*}
where $f_3=\frac{1}{x^0u^0}G(u_1,u_1)-\frac{1}{u^0}D_{u_1}u^0$  is a
higher order term. From the Bernoulli law,
\begin{eqnarray*}
\p_0(u^0)^2=-\frac{2}{\gamma-1}\p_0(c^2)+\frac{2\varphi_3}{u^0}+f_4(U,DU).
\end{eqnarray*}
By the equation of state,
\begin{eqnarray*}
\p_0(c^2)=(\gamma-1)p^{-\frac{1}{\gamma}}A(S)^{\frac{1}{\gamma}}\p_0
p
+p^{\frac{\gamma-1}{\gamma}}A(S)^{\frac{1-\gamma}{\gamma}}\p_0A(S).
\end{eqnarray*}
However, the invariance of entropy implies that
\begin{eqnarray*}
\p_0A(S)=\frac{\varphi_4}{u^0}+f_5(U,DU).
\end{eqnarray*}
Therefore, we have
\begin{eqnarray}\label{170}
A(S)(x^0,x)=\psi^*A(S)+\int_{\psi(x)}^{x^0}\frac{\varphi_4}{u^0}(s,x)\,\dd
s+f_6(U,\psi,DU),
\end{eqnarray}
and
\begin{eqnarray*}
&&c^2=\gamma p^{\frac{\gamma-1}{\gamma}}\big(\psi^*
A(S)\big)^{\frac{1}{\gamma}}+L_1(\varphi_4)+f_7(U,\psi,DU),\\
&&(u^0)^2=2E_0(x)-\frac{2\gamma}{\gamma-1}p^{\frac{\gamma-1}{\gamma}}
\big(\psi^*A(S)\big)^{\frac{1}{\gamma}}+L_2(\varphi_3,\varphi_4)+f_8(U,\psi,DU),
\end{eqnarray*}
where $L_1$ and $L_2$ do not involve the derivatives of $\varphi_i$, and
$L_1(0)=L_2(0,0)=0$.   We then have
\begin{eqnarray}\label{172}
&&\rho x_0^2\Big(\di
\varphi_0-D_u\bar{\varphi}_1-\frac{2u^0}{x^0}\bar{\varphi}_1+L_3(dx^0(\varphi_0),
\varphi_3,\varphi_4)\Big)\nonumber\\
&&= x_0^2\big(1-\frac{(u^0)^2}{c^2}\big)\p_0^2p
-\Delta'p+2x_0\big(2-\frac{(u^0)^2}{c^2}\big)\p_0p\nonumber\\&&
\quad+\frac{x_0^2}{p}\big(\frac{(u^0)^2}{c^2}+\frac
{c^2}{\gamma(u^0)^2}\big)(\p_0p)^2+\frac{4\gamma}{\gamma-1}p-4E_0(x)p^{\frac{1}
{\gamma}}\big(\psi^*A(S)\big)^{-\frac{1}{\gamma}}\nonumber\\
&&=e_1\,\p_0^2\hat{p}-\Delta'\hat{p}+e_2\, \p_0
\hat{p}+e_3\,\hat{p}+e_4\,\psi^*\hat{p}+f_{9}(U,U^-,\psi,DU,D\psi,D^2U)
\end{eqnarray}
by \eqref{144} and \eqref{146}, where $L_3(0,0,0)=0$ and
$e_i=e_i(x^0)$, $i=1,\cdots,4$, are known functions determined by
the background solution:
\begin{eqnarray*}
e_1&=&(x^0)^2(1-t)>0,\\
e_2&=&\frac{2x^0}{1-t}\big((1+2\gamma)t^2-3t+4\big)>0,\\
e_3&=&\frac{-2}{(t-1)^3}\big(6-19t-7t^2(-2+\gamma)+t^4\gamma(1+2\gamma)
  +t^3(-3+2\gamma-4\gamma^2)\big),\\
%\cdot s_1(t,\gamma),\\
e_4&=&\frac{\mu_4}{\mu_2}\,\frac{2\rho^\gamma(2+(\gamma-1)t)}{(\gamma-1)
(t-1)^3}\, \big((2\gamma-3)t^2+8t-3\big)\nonumber\\
&=& \frac{1-t_s}{\big(\rho_b^+(r_b)\big)^\gamma(1+(\gamma-1)t_s)}\,
\frac{2\rho^\gamma(2+(\gamma-1)t)}{(1-t)^3}\,
\big((2\gamma-3)t^2+8t-3\big),
\end{eqnarray*}
where
$t=t(x^0)
=(M_b^+)^2(x^0)\in (0,1)$ and
$t_s=t(r_b)=\Big(\frac{(u^0)_b^+(r_b)}{c_b^+(r_b)}\Big)^2=(M_b^+)^2(r_b)\in
(0,1)$ with $M_b^+(x^0)=\frac{(u^0)_b^+(x^0)}{c_b^+(x^0)}$.
Since $e_1>0$, \eqref{172} is {\it an elliptic equation} for
$\hat{p}$.

We also recall that $t(x^0)$ is monotonically decreasing for the
background solution  and satisfies the following differential
equation (cf. \cite{Yu3}):
\begin{equation}\label{b16}
\frac{\dd t}{\dd x^0}=\frac{2t}{x^0}\,\frac{2+(\gamma-1)t}{t-1}.
\end{equation}

\subsection{The Normalization of $\mathcal{M}_\psi^+$ and Reduced Equations}

The above equations and boundary conditions are obtained in
$\mathcal{M}_\psi^+$ for the given subsonic flow $U$ and the
shock-front $\psi$ satisfying Theorem \ref{thm101}. The computations
are relatively easy for the sake of rather simple metric $G$.
However, to 
show the uniqueness of the transonic shock-front and the subsonic
flow behind it,  we need to set up an iteration mapping to find a new
$\hat{\psi}\in \mathcal{K}_\sigma$:

\begin{eqnarray} \label{258}
\mathcal{K}_\sigma:=\left\{\psi\in C^{4,\alpha}(\mathbf{S}^2)\,:\,
\norm{\psi-r_b}_{C^{4,\alpha}(\mathbf{S}^2)}\le
\sigma\le\sigma_0\right\}
\end{eqnarray}
for a positive constant $\sigma_0$ to be specified later, by
solving several boundary value problems in $\mathcal{M}_\psi^+$ for
any $\psi\in \mathcal{K}_\sigma$, and then show that there is a fixed
point that is unique. Therefore, it is convenient to
introduce a $C^{4,\alpha}$--homeomorphism $ \Psi:
(x^0,x)\in\mathcal{M}_\psi^+\mapsto ({y}^0,y)\in
\Omega:=[0,1]\times\mathbf{S}^2$ by
\begin{eqnarray}\label{259}
{y}^0=\frac{x^0-\psi(x)}{r^1-\psi(x)},\qquad y=(y^1, y^2)=(x^1,
x^2),
\end{eqnarray}
to normalize $\mathcal{M}_\psi^+$ to $\Omega$. We set
$\Omega^{j}=\{j\}\times\mathbf{S}^2, j=0,1$. Then
$\p\Omega=\Omega^0\cup\Omega^1$. We will use $i$ to denote the
embedding of $\Omega^0$ in $\Omega$.
We also define the metric of $\Omega$ to be
$$
\tilde{G}
=(r^1-r_b)^2dy^0\otimes dy^0+\big((r^1-r_b)y^0+r_b\big)^2g,
$$
which
differs from the metric induced by $\Psi$ only  those terms
involving $D\psi$ or $\psi-r_b$. Therefore, according to
\eqref{222}, we define, for a vector field $u$ in $\Omega$, the
corresponding $\bar{u}^1=u^\alpha\tilde{G}_{\alpha\beta}dy^\beta$.
In the following, $\p_i=\frac{\p}{\p y^i}$ in $\Omega$ for short.
Then \eqref{185} can be written on $\Omega^0$ as
\begin{eqnarray}\label{261}
d\bar{u}|_{\Omega^0}&=&d(\bar{u}|_{\Omega^0})+E_1(U,\psi,DU,D\psi,D^2\psi)|_{\Omega^0}\wedge
d\psi\nonumber\\
&&+(r^1-r_b)\Big((r^1-r_b)r_bg_{\alpha\beta}u^\beta|_{\Omega^0}
 -\left. \frac{\p_\alpha p}{\rho_b^+(u^0)_b^+}\right|_{\Omega^0}\Big)dy^0\wedge dy^\alpha\nonumber\\
&&+g_9(U,\psi,DU,D\psi) +\left.
(r^1-r_b)\,\psi^2g_{\alpha\beta}\frac{dy^\alpha(\varphi_0)}{u^0}\right|_{\Omega^0}dy^0\wedge
dy^\beta,
\end{eqnarray}
where $E_1$ is a 1-form in $\Omega$ depending smoothly on
$U,\psi,DU,D\psi$, and $D^2\psi$. Equations \eqref{124}--\eqref{125}
on $\Omega^0$ are respectively
\begin{eqnarray}
&&d\psi=\omega:=\mu_0i^*(\bar{u}^1)+\bar{g}_0(U,U^-,\psi,D\psi)\in
A^1(\mathbf{S}^2), \label{262}\\
\label{263}
&&d(i^*\bar{u}^1)=\chi(U,U^-,\psi):=-\frac{d\bar{g}_0}{\mu_0}.
\end{eqnarray}
Similarly, \eqref{143}--\eqref{146} are transferred to
\begin{eqnarray}
&&i^*(\widehat{u^0})=\bar{\mu}_1\,(\psi^p+r^p-r_b)+\bar{g}_1(U,U^-,\psi,D\psi),\label{264}\\
&&i^*(\hat{p})=\mu_2\,(\psi^p+r^p-r_b)+\bar{g}_2(U,U^-,\psi,D\psi),\label{265}\\
&&i^*(\hat{\rho})=\mu_3\,(\psi^p+r^p-r_b)+\bar{g}_3(U,U^-,\psi,D\psi),\label{266}\\
&&i^*(\widehat{A(S)})=\mu_4\,(\psi^p+r^p-r_b)+\bar{g}_4(U,U^-,\psi,D\psi),\label{267}
\end{eqnarray}
where $\bar{\mu}_1=\frac{\mu_1}{r^1-r_b}>0$. In addition, we have
\begin{eqnarray}\label{268}
\eqref{156}&\Longleftrightarrow&
d^*(i^*\bar{u}^1)=\frac{\mu_5}{r^1-r_b} i^*(\p_0\hat{p})+\mu_6\,
\psi^p+\mu_6\,(r^p-r_b)\nonumber\\
&&\qquad\qquad\,\,\,\, +\bar{g}_5(U,U^-,\psi,DU,D\psi)+
\psi^2i^*\Big((r^1-r_b)\frac{dy^0(\varphi_0)}{(u^0)_b^+}-\bar{\varphi}_1\Big),\\
\label{269} \eqref{psieq}&\Longleftrightarrow &
\Delta'\psi^p+\mu_7\psi^p=\mu_0\mu_6(r^p-r_b)+
\frac{\mu_0\mu_5}{r^1-r_b}i^*(\p_0\hat{p})+\bar{g}_6(U,U^-,\psi,DU,D\psi,D^2\psi)\nonumber\\
&&\qquad\qquad\quad\qquad
+\mu_0\psi^2i^*
\big((r_1-r_b)\frac{dy^0(\varphi_0)}{(u^0)_b^+}-\bar{\varphi}_1\big),\qquad\\
\label{270} \eqref{req} &\Longleftrightarrow &
r^p-r_b=-\frac{1}{4\pi\mu_6}\int_{\mathbf{S}^2}\Big(\frac{\mu_5}{r^1-r_b}i^*(\p_0\hat{p}
)+\bar{g}_5(U,U^-,\psi,DU,D\psi)\Big)\, \vol,\\
\eqref{177} &\Longleftrightarrow &
\psi^p=\frac{1}{\mu_2}\Big(i^*\hat{p}-\frac{\mu_8}{r^1-r_b}
\int_{\mathbf{S}^2}i^*(\p_0\hat{p})\,\vol+{\bar{g}_7(U,U^-,\psi,D\psi)}\Big),\label{271}\\
\eqref{178}&\Longleftrightarrow&
\Delta'(i^*\hat{p})+\mu_7(i^*{\hat{p}})+\frac{\mu_9}{r^1-r_b}\,i^*(\p_0\hat{p})
=\bar{g}_8(U,U^-,\psi,DU,D\psi,D^2U,D^2\psi,D^3\psi)\nonumber\\
&&\qquad\qquad\qquad\qquad\qquad\qquad\qquad\qquad
  +\mu_2\mu_0\psi^2i^*\big((r^1-r_b)\frac{dy^0(\varphi_0)}{(u^0)_b^+}
-\bar{\varphi}_1\big),\qquad \label{272}\\
\label{273} \eqref{172}&\Longleftrightarrow& (r^1-\psi)^2\rho
x_0^2\Big(\di
\varphi_0-D_u\bar{\varphi}_1-\frac{2(r^1-r_b)(u^0)_b^+}{(r^1-r_b)y^0+r_b}
\,\bar{\varphi}_1+L_3((r^1-r_b)\,dy^0(\varphi_0),
\varphi_3,\varphi_4)\Big)\nonumber\\
&&\qquad=e_1\,\p_0^2\hat{p}-(r^1-r_b)^2\Delta'\hat{p}+(r^1-r_b)e_2\,\p_0
\hat{p}+(r^1-r_b)^2e_3\,\hat{p}+(r^1-r_b)^2e_4\,
i^*\hat{p}\nonumber\\
&&\qquad\,\,\,\,\,-\bar{f}_{9}(U,U^-,\psi,DU,D\psi,D^2U),
\end{eqnarray}
where $\bar{f}_j, \bar{g}_j$ differ from $f_j, g_j$ by some higher
order terms due to the facts that $\tilde{G}$ differs from
$(\Psi^{-1})^*G$ by the terms involving $D\psi$ or $\psi-r_b$ and
that $\hat{U}$ is small. We also note that
$e_i=e_i((r^1-r_b)y^0+r_b)$ in \eqref{273}.

\begin{remark}\label{rm-2.2}
An important observation is that, by \eqref{107}, if $\varphi_4=0$,
we may write
$$
\varphi_3=G(\varphi_0,u)=(r^1-r_b)\cdot u_b^+((r^1-r_b)y^0+r_b)\cdot
dy^0(\varphi_0)+\text{h.o.t.},
$$
so we may write $L_3((r^1-r_b)\, dy^0(\varphi_0),\varphi_3, 0)$ in
\eqref{273} as an expression of $dy^0(\varphi_0)$ with coefficients
depending only on $y^0$ by adjusting the higher order term
$\bar{f}_9$.
\end{remark}

\subsection{The S-Condition}
We now state the {\it S-condition} assumed in our main theorem.

Consider the following boundary value problem:
\begin{eqnarray}
&&e_1v''+(r^1-r_b)e_2v'+(r^1-r_b)^2(e_3-\lambda_n)v
=-(r^1-r_b)^2e_4,\label{c4}\\
&&v(0)=1,\qquad v(1)=0,\label{c5}\\
&&v'(0)=-\frac{\lambda_n+\mu_7}{\mu_9}(r^1-r_b), \label{c6}
\end{eqnarray}
where $\lambda_n=n(n+1), n=0,1,2,\cdots$, and $e_i=e_i(t(y^0))$ with
$t(y^0)=(r^1-r_b)y^0+r_b$ are considered as functions of $y^0$ on
$[0,1]$.

\begin{definition}
A background solution $U_b$ satisfies the {\it S-Condition} if, for
each $n=0,1,2,\cdots$,  problem \eqref{c4}--\eqref{c6} does not have
a solution.
\end{definition}

The following lemmas show that there exist certain background
solutions which satisfy the S-Condition.

\begin{lemma}\label{lemc01}
For given $\gamma>1$, let $U_b$ be a background solution determined
by the supersonic upstream data $p_b^-(r^0),\rho_b^-(r^0),
t_b^-(r^0)=(M_b^-)^2(r^0)$, and the back pressure $p_b^+(r^1)$,
where $M_b^-(r^0)$ is the Mach number of the upstream supersonic
flow. Then, for given $\rho_b^-(r^0)$ and the back pressure
$p_b^+(r^1)$, when $p_b^-(r^0)$ and $M_b^-(r^0)$ are sufficiently
large, $U_b$ satisfies the S-Condition.
\end{lemma}

\begin{proof} We divide the proof into three steps.

\medskip
{\it Step 1}. By an analysis of the background solution in
\cite{Yu3}, it suffices to show that, if $\kappa=r^1-r_b>0$ and
$\sigma=t_b^+(r_b)=(M_b^+)^2(r_b)$ are rather small, then $U_b$
satisfies the S-Condition. We will prove by contradiction: we will
first assume that \eqref{c4}--\eqref{c6} has a solution and then
lead to a contradiction.

We note that, once $t_b^-(r^0)$ is given, we can solve $t(x^0)$ for
all $x^0\in[r^0,r^1]$ (i.e., it is independent of $p_b^-(r^0),
\rho_b^-(r^0)$, and $p_b^+(r^1)$). This can be seen from \eqref{b16}
and the Rankine-Hugoniot condition of the Mach number:
\begin{eqnarray*}
\Big(\frac{1}{\sqrt{t_b^-(r_b)}}+\frac{\gamma-1}{2}\Big)
\Big(\frac{1}{\sqrt{t_b^+(r_b)}}+\frac{\gamma-1}{2}\Big)=\frac{(\gamma+1)^2}{4}.
\end{eqnarray*}
Moreover, since $p_b^+(r^1)$ is given,  we have the estimate:
\begin{eqnarray*}
1\le \frac{p_b^+(x^0)}{p_b^+(r_b)}\le C\qquad \text{for any} \quad
x^0\in(r_b,r^1),
\end{eqnarray*}
with $C$ depending only on $p_b^+(r^1)$, $r^0$, and $r^1$.

\medskip
{\it Step 2.}  We choose $\sigma_0$ such that, for $t\le\sigma_0$,
$e_4(t)\le0$. Hence, for any $t_b^+(r_b)\le\sigma_0$ (this
requires $t_b^-(r^0)$ large), by \eqref{b16},
$$
e_4(t(y^0))\le0  \qquad \mbox{for $y^0\in(0,1)$}.
$$
In addition, if $n\ge3,$ we see that $e_3-\lambda_n<0$
and $\frac{(r^1-r_b)(\lambda_n+\mu_7)}{\mu_9}<0$. So, by the Hopf
maximum principle, we infer a contradiction if $v(1)=0$ for a
solution $v$ to \eqref{c4}--\eqref{c6}.

For $n=0,1,2$, we will utilize an energy estimate below to obtain a
contradiction if $\kappa$ is also small.

\medskip
{\it Step 3.} We first reformulate \eqref{c4}--\eqref{c6}. Let
$h_n=-\frac{\mu_7+\lambda_n}{\mu_9}\kappa$. For simplicity, we write
the independent variable $y^0$ as $y$. Then, by multiplying
\begin{eqnarray*}
p_n(y)=\exp\left(\int_0^y\big(2h_n+\kappa\frac{e_2(t(s))}{e_1(t(s))}\big)\,\dd
s\right)
\end{eqnarray*}
to \eqref{c4}, we see that $w=e^{-h_ny}v$ satisfies
\begin{eqnarray*}
p_n(y)\frac{\dd (p_n(y)\frac{\dd w}{\dd y})}{\dd
y}+\kappa^2\alpha_n(y)w(y)=\kappa^2\beta_n(y),
\end{eqnarray*}
where
\begin{eqnarray*}
&&\alpha_n(y)=p_n(y)^2\Big(\frac{(\lambda_n+\mu_7)^2}{\mu_9^2}
 -\frac{\lambda_n+\mu_7}{\mu_9}\,\frac{e_2(t(y))}{e_1(t(y))}
  +\frac{e_3(t(y))-\lambda_n}{e_1(t(y))}\Big),\\
&&\beta_n(y)=-\big(p_n(y)\big)^2e^{-h_ny}\frac{e_4(t(y))}{e_1(t(y))}.
\end{eqnarray*}
By a change of the independent variable $y\rightarrow z_n$:
\begin{eqnarray*}
z_n=\int_0^y\frac{\dd s}{p_n(s)} \qquad \text{with}\quad
z_n^*=\int_0^1\frac{\dd s}{p_n(s)},
\end{eqnarray*}
the above equation becomes
\begin{eqnarray}\label{energy}
w''+\kappa^2\alpha_n w=\kappa^2\beta_n,
\end{eqnarray}
where $'$ is the derivative with respect to $z_n$,
$\alpha_n=\alpha_n(y(z_n))$, and $\beta_n=\beta_n(y(z_n))$. The
boundary conditions are
\begin{eqnarray}\label{c21}
w(0)=1,\qquad w'(0)=0,\qquad w(z_n^*)=0.
\end{eqnarray}

Now, multiplying $w$ to \eqref{energy} and integrating on
$[0,z_n^*]$ yield
\begin{eqnarray}\label{c22}
\int_0^{z_n^*}(w')^2\,\dd z =-\kappa^2\int_0^{z_n^*}\beta_nw\,\dd
z+\kappa^2\int_0^{z_n^*}\alpha_nw^2\,\dd z.
\end{eqnarray}
Note here that, since $n<3$, and $|\beta_n|, |\alpha_n|$, and
$z_n^*$ are bounded by a constant $C$,
\begin{eqnarray*}
&&\Big|\int_0^{z_n^*}\beta_n w(z)\,\dd z\Big|
 \le C\int_0^{z_n^*}\Big|\int_{z_n^*}^{z}w'(s)\,\dd s\Big|\,\dd z
  \le 2C\sqrt{z_n^*}\sqrt{\int_0^{z_n^*}(w')^2\,\dd z},\\
&&\Big|\int_0^{z_n^*}\alpha_n w(z)^2\,\dd z\Big|
 \le C\int_0^{z_n^*}\Big(\int_{z_n^*}^{z}w'(s)\,\dd s\Big)^2\,\dd z
  \le C\frac{z_n^*}{2}\int_0^{z_n^*}(w')^2\,\dd z.
\end{eqnarray*}
By \eqref{c21}, $\int_0^{z_n^*}(w')^2\,\dd z\ne0$. Then, from
\eqref{c22}, we have
\begin{eqnarray*}
\int_0^{z_n^*}(w')^2\,\dd z\le \frac{C'\kappa^2}{1-C'\kappa^2}\le
C''\kappa^2.
\end{eqnarray*}
On the other hand,
\begin{eqnarray*}
1=\Big|\int_0^{z_n^*}w'\,\dd z\Big|^2\le
z_n^*\int_0^{z_n^*}(w')^2\,\dd z\le C'''\kappa^2.
\end{eqnarray*}
This reaches a  contradiction when
$\kappa<\min\{\frac{1}{\sqrt{C'''}}, \frac{1}{\sqrt{2C'}}\}$. Note
that $C'''$ depends only on $r^0,r^1$, $t_b^-(r^0)$, $n$, and
$p_b^+(r^1)$.
\end{proof}

\begin{remark} In Lemma \ref{lemc01}, for the case that
$\kappa=r^1-r_b>0$ is small, we require only $p_b^-(r_0)$ to be
large (see \cite{Yu3}).
\end{remark}

\medskip

In the following, we provide some other results on the existence of
background solutions that satisfy the S-Condition. For given
$[r^0,r^1]$, we note that a background solution $U_b^+$ is
determined by the five parameters $(\gamma, r_b, p_b^+(r_b),
\rho_b^+(r_b), t(r_b))$ with $t(r_b)=(M_b^+)^2(r_b)$, $\gamma>1,\
r_b\in(r^0,r^1), \ p_b^+(r_b)>0, \ \rho_b^+(r_b)>0$, and
$t(r_b)\in(0,1)$.

\begin{lemma}
For given $\gamma>1, \ \rho_b^+(r_b)>0$, and $\sigma_0\in(0,1)$,
there exist a set $S_1\subset[0,\sigma_0]$ of at most countable
infinite points and a set $S_2\in[r^0,r^1]$ of at most finite points
such that the background solution determined by $\gamma>1,\
r_b\in(r^0,r^1)\setminus S_2, \ p_b^+(r_b)>0,\ \rho_b^+(r_b)>0$, and
$t(r_b)\in(0,\sigma_0)\setminus S_1$ satisfies the S-Condition.
\end{lemma}

\begin{proof}
Let $\sigma=t(r_b)\in[0,\sigma_0]$. Then, by \eqref{b16}, we know
that $t=t(x^0,\sigma)$ is analytical with respect to $\sigma$. Also,
by the theory of ordinary differential equations, $v=v(y^0,\sigma)$
is an analytical function of $\sigma$. Hence, we establish an
analytical mapping $f_n: [0,\sigma_0]\rightarrow \mathbb{R}$ by
\begin{eqnarray}
f_n(\sigma)=v(1,\sigma),\qquad n=0,1,2,\cdots
\end{eqnarray}

Now, if $f_n^{-1}(\{0\})$ has infinite points, then the zeros of
$f_n$ have an accumulate point in $[0,\sigma_0]$, which implies
$f_n\equiv 0$, especially, $f_n(0)=0$.

Note that $t(x^0,0)\equiv 0$. Thus, in this case,
\eqref{c4}--\eqref{c6} become
\begin{eqnarray}
&&((r^1-r_b)y^0+r_b)^2v''+8(r^1-r_b)\big((r^1-r_b)y^0+r_b\big)v'
  +(r^1-r_b)^2(12-\lambda_n)v\nonumber\\
&&\qquad=12(r^1-r_b)^2\frac{\rho_b^+(x^0)^\gamma}{\rho_b^+(r_b)^\gamma}\ge0,\label{c8}\\
&&v(0)=1,  \label{c9}\\
&&v'(0)=\frac{(r^1-r_b)\lambda_n}{2r_b}\ge0,\label{c10}\\
&& v(1)=0.\label{c11}
\end{eqnarray}

For $n\ge 3$, $(r^1-r_b)^2(12-\lambda_n)\le0$. By the Hopf maximum
principle, we infer a contradiction at $y^0=0$. Therefore, in these
cases, $f_n^{-1}(\{0\})$ consists of finite points, which implies
that $\cup_{n=3}^\infty f_n^{-1}(\{0\})$ is countable.

For $n=0,1,2$, we also recognize that the solution $v$ of problem
\eqref{c8}--\eqref{c10} is an analytical function of the parameter
$r_b$. Hence, we have an analytical mapping $g_n:
[r^0,r^1]\rightarrow \mathbb{R}$ defined by $g_n(r_b)=v(1,r_b)$. We
claim that $\cup_{n=0}^2g_n^{-1}(\{0\})$ has only finite points.
Indeed, if $g_n^{-1}(\{0\})$ contains infinite points, then
$g_n\equiv0$, especially $g_n(r^1)=0$. However, in this case,
\eqref{c8}--\eqref{c11} are reduced to
\begin{equation}\label{c12}
\left\{\begin{array}{ll}
v''=0,\\
v(0)=1,\quad v(1)=0, \quad v'(0)=0.
\end{array}\right.
\end{equation}
Obviously, there is no solution to this problem.
\end{proof}

\begin{lemma}
For given $\gamma>1, \rho_b^+(r_b)>0$, and $t(r_b)\in(0,1)$, there
exists a set $S_3\subset[r^0,r^1]$ of at most countable infinite
numbers of points such that the background solution determined by
$\gamma,\ r_b\in(r^0,r^1)\setminus S_3,\ p_b^+(r_b)>0,
\rho_b^+(r_b)$, and $t(r_b)$ satisfies the S-Condition.
\end{lemma}

\begin{proof}
As in the above proof, consider $v(1)$ as an analytical function of
$r_b$ for each $n$. If the pre-image of the zero has infinite points
for some $n$, then the function is identically zero and, by
\eqref{c12}, there is a contradiction.
\end{proof}

\begin{lemma}
There exists $\sigma_0\in(0,1)$ such that, for given $\gamma>1,
p_b^+(r_b)>0, \rho_b^+(r_b)>0$, and
$t(r_b)=(M_b^+)^2(r_b)\in(0,\sigma_0)$, there exists
$r_*\in(r^0,r^1)$ so that the background solution, determined by
$\gamma>1, r_b\in(r_*,r^1), p_b^+(r_b), \rho_b^+(r_b)$, and
$t(r_b)$, satisfies the S-Condition.
\end{lemma}

\begin{proof}
We choose $\sigma_0$ such that, for $t\le \sigma_0$,
$e_4(t)\le0$. Hence, for any $t(r_b)\le\sigma_0$,
$$
e_4(t(y^0))\le 0 \qquad \mbox{for $y^0\in(0,1)$}
$$
by \eqref{b16}. In addition, if
$n\ge 3$, we see that
$$
e_3-\lambda_n<0, \qquad \frac{(r^1-r_b)(\lambda_n+\mu_7)}{\mu_9}<0.
$$
Thus, by the Hopf
maximum principle, we infer a contradiction if $v(1)=0.$

Now, for $n=0,1,2$, consider $K=\cup_{n=0}^2g_n^{-1}(\{0\})$ which
have at most finite points. Note that $r^1\notin K$. Thus, let
$r_*=\sup K <r^1$ (if $K=\emptyset$, let $r_*=r^0$). Then the lemma
is proved.
\end{proof}

This lemma improves somewhat the results of Lemma \ref{lemc01}, but
we do not have an estimate of $r^1-r_*$ here as that in Lemma
\ref{lemc01}.

\section{An Iteration Mapping and Decomposition of the Euler System}

In this section we set up an iteration mapping and show that it has
a unique fixed point. By the derivations in \S 2, it is easy to see
that any transonic shock solution to \eqref{101}--\eqref{103}
satisfying those requirements in Theorem \ref{thm101} must be a
fixed point of the iteration mapping, which implies its
uniqueness claimed in Theorem \ref{thm101}.
Another motivation to introduce the iteration mapping is for constructing
approximate solutions to show the existence of global transonic shock solutions, which
require further exploration.

\subsection{The Iteration Set}

For given $\psi\in \mathcal{K}_\sigma$, its position $r^p$ and
profile $\psi^p$ satisfy
\begin{eqnarray*}
\norm{\psi^p}_{C^{4,\alpha}(\mathbf{S}^2)}\le 2\sigma,\quad
|{r^p-r_b}|\le\sigma.
\end{eqnarray*}
We solve the candidate subsonic flow in
$\mathcal{M}_\psi^+:=\{(x^0,x)\in\mathcal{M}:x^0\ge\psi(x),
x\in\mathbf{S}^2\}$. By \eqref{259}, we write the set of possible
variations of the subsonic flows as
\begin{eqnarray}
O_\delta&:=&\Big\{\check{U}=(\check{p},\check{\rho},\check{u})\,:\,
\norm{(\check{p},\check{\rho})}_{C^{3,\alpha}(\Omega)}
+\norm{\bar{\check{u}}}_{C^{3,\alpha}_1(\Omega)}\le\delta\le\delta_0\Big\}
\end{eqnarray}
with constant $\delta_0$ to be chosen later. Given $U^-$ satisfying
\eqref{015}, for any $\psi\in \mathcal{K}_\sigma$ and $\check{U}\in
O_\delta$, we construct a mapping $\mathcal{K}_\sigma\times
O_\delta\rightarrow \mathcal{K}_\sigma\times O_\delta$ denoted as
$$
\mathcal{T}(\psi,\check{U})=(\hat{\psi},\hat{U})
$$
by the following
iteration process. Then
we show that $\mathcal{T}$ has a unique fixed point in
$\mathcal{K}_\sigma\times O_\delta.$

\subsection{A Nonlocal Venttsel Problem for the Candidate Pressure in
$\Omega$}

We first choose $\varepsilon_0$, $\sigma_0$, and $\delta_0$ small
enough such that the formulations in \S 2 valid. For any $\psi\in
\mathcal{K}_\sigma,\ \check{U}\in O_\delta$, and $U^-$ satisfying
\eqref{015}, we may express the higher order terms $\bar{f}_i$ and
$\bar{g}_i$ in terms of $U=U_b^++\check{U}$.

By \eqref{th103} and \eqref{272}--\eqref{273}, we solve $\hat{p}$
from the following linear nonlocal Venttsel problem:
\begin{eqnarray}
&&e_1\p_{0}^2\hat{p}-(r^1-r_b)^2\Delta'\hat{p}+(r^1-r_b)e_2\p_{0}
\hat{p}+(r^1-r_b)^2e_3\hat{p}+(r^1-r_b)^2e_4\,\hat{p}|_{\Omega^0}
\nonumber\\
&&\qquad =\bar{f}_{9}(U,U^-,\psi,DU,D\psi,D^2U)\qquad\,\,
\qquad\qquad
 \text{in}\ \ \Omega,\label{175}\\
&&\hat{p}=0 \quad\qquad
\qquad\qquad\qquad\qquad\qquad\qquad\qquad\qquad
 \text{on}\ \ \Omega^1,\label{176}\\
&&\Delta'(\hat{p}|_{\Omega^0})+\frac{\mu_9}{r^1-r_b}
\p_{{0}}\hat{p}|_{\Omega^0}
+\mu_7\,
{\hat{p}}|_{\Omega^0}=\bar{g}_8(U,U^-,\psi,DU,D\psi,D^2U,D^2\psi,D^3\psi)
  \quad\text{on}\ \Omega^0.\qquad\label{309}
\end{eqnarray}

Thanks to Theorem 1.5 in \cite{LN} for the Venttsel problem (note
that $\frac{\mu_9}{r^1-r_b}<0$) and Theorem 6.6 in \cite{GT} for the
Dirichlet problem, with the aid of a standard higher regularity
argument as in Theorem 6.19 of \cite{GT}, by considering
$(r^1-r_b)^2e_4\, \hat{p}|_{\Omega^0}$ as a nonhomogeneous term and
using interpolation inequalities, for the solution $\hat{p}\in
C^{3,\alpha}(\Omega)$, we have the apriori Schauder estimate
\begin{eqnarray}\label{310}
\norm{\hat{p}}_{C^{3,\alpha}(\Omega)}\le
C_2\Big(\norm{\hat{p}}_{C^0(\Omega)}+\norm{\bar{f}_{9}}_{C^{1,\alpha}(\Omega)}
+\norm{\bar{g}_8}_{C^{1,\alpha}(\Omega^0)}\Big)
\end{eqnarray}
with constant $C_2$ depending only on $U_b$.

Next, by Lemma \ref{lema5a},
let $u_{n,m}(y)$ be the eigenfunctions
of $\Delta'$ on $\mathbf{S}^2$ with respect to the eigenvalues
$\lambda_n=n(n+1)\ge 0, n=0,1,2,\cdots$.
Then
$$
\hat{p}=\sum_{n=0}^\infty\sum_{m=1}^{2n+1} v_{n,m}(y^0)u_{n,m}(y).
$$
For $\bar{f}_{9}=0$ and $\bar{g}_8=0$, each $v_{n,m}$ satisfies the
nonlocal differential equation:
\begin{eqnarray}
&&e_1v_{n,m}''+(r^1-r_b)e_2v_{n,m}'+(r^1-r_b)^2(e_3-\lambda_n)v_{n,m}
+(r^1-r_b)^2e_4v_{n,m}(0)=0,\label{c-1}\\
&&v_{n,m}(1)=0,\label{c-2}\\
&&\mu_9v_{n,m}'(0)+(r^1-r_b)(\lambda_n+\mu_7)v_{n,m}(0)=0.\label{c-3}
\end{eqnarray}
 
First, if $v_{n,m}(0)=0,$ \eqref{c-3} says that $v'_{n,m}(0)=0$.
Hence, by uniqueness of solutions of the Cauchy problem of
differential equations, we infer $v_{n,m}\equiv 0$.

Now, if $v_{n,m}(0)\ne0$, then, by considering
$\frac{v_{n,m}(y^0)}{v_{n,m}(0)}$ as the unknown, we see that it
solves
\begin{eqnarray}
&&e_1v''_{n,m}+(r^1-r_b)e_2v'_{n,m}+(r^1-r_b)^2(e_3-\lambda_n)v_{n,m}
=-(r^1-r_b)^2e_4,\label{c-4}\\
&&v_{n,m}(0)=1,\qquad v_{n,m}(1)=0,\label{c-5}\\
&&v'_{n,m}(0)=-\frac{\lambda_n+\mu_7}{\mu_9}(r^1-r_b).\label{c-6}
\end{eqnarray}
The $S$-condition in \S 2.6  guarantees the nonexistence of a solution
to this problem. Thus, $v_{n,m}(0)=0$, which implies $\hat{p}\equiv 0$.

Therefore, the S-Condition implies the uniqueness of solutions of
the Venttsel problem \eqref{175}--\eqref{309}. Then, by \eqref{310}
and a standard argument of contradiction based on the compactness
(cf. Lemma 9.17 in \cite{GT}), we have the apriori estimate:
\begin{eqnarray}\label{314}
\norm{\hat{p}}_{C^{3,\alpha}(\Omega)}\le
C_2\Big(\norm{\bar{f}_{9}}_{C^{1,\alpha}(\Omega)}
+\norm{\bar{g}_8}_{C^{1,\alpha}(\Omega^0)}\Big)
\end{eqnarray}
for any $C^{3,\alpha}$ solution of the above nonlocal Venttsel
problem. Then, by the method of continuity as carried out in
\cite{LN}, we see that this problem has a unique solution
$\hat{p}\in C^{3,\alpha}(\Omega)$ satisfying estimate
\eqref{314}.

\subsection{Update of the Candidate Free Boundary}

Once we get $\hat{p}$, according to \eqref{270}--\eqref{271}, we may
obtain the new profile of the free boundary $\hat{\psi}^p$ and the
position $\hat{r}^p$ by
\begin{eqnarray}
&&\hat{\psi}^p=\frac{1}{\mu_2}\Big(\hat{p}|_{\Omega^0}-\frac{\mu_8}{r^1-r_b}
\int_{\mathbf{S}^2}\p_{{0}}\hat{p}|_{\Omega^0}\,\vol
+{\bar{g}_7(U,U^-,\psi,D\psi)}\Big),\label{315}\\
&&\hat{r}^p-r_b=-\frac{1}{4\pi\mu_6}\int_{\mathbf{S}^2}
\Big(\frac{\mu_5}{r^1-r_b}\, \p_{{0}}\hat{p}|_{\Omega^0}
+\bar{g}_5(U,U^-,\psi,DU,D\psi)\Big)\, \vol.
\end{eqnarray}
In fact, one may show that $\int_{\mathbf{S}^2}\hat{\psi}^p\,\vol=0$
by \eqref{315} and \eqref{309}; However, we do not need this in the
process. We need to improve the regularity of $\hat{\psi}$. By
\eqref{309} and \eqref{315}, $\hat{\psi}^p$ satisfies this elliptic
equation on $\mathbf{S}^2:$
\begin{eqnarray}\label{317}
\Delta'\hat{\psi}^p+\mu_7\hat{\psi}^p=\mu_0\mu_6(\hat{r}^p-r_b)
+\frac{\mu_0\mu_5}{r^1-r_b}\,\p_0\hat{p}|_{\Omega^0}
+\bar{g}_6(U,U^-,\psi,DU,D\psi,D^2\psi).
\end{eqnarray}
The right-hand side belongs to $C^{2,\alpha}(\mathbf{S}^2)$. Thus,
by Theorem 6.19 in \cite{GT}, $\hat{\psi}=\hat{\psi}^p+\hat{r}^p$
obviously obeys the estimate:
\begin{eqnarray}\label{318a}
\|\hat{\psi}-r_b\|_{C^{4,\alpha}(\mathbf{S}^2)}&\le&
C_2\Big(\norm{(\bar{g}_5,\bar{g}_7)}_{C^0(\Omega^0)}
+\norm{\bar{f}_{9}}_{C^{1,\alpha}(\Omega)}+\norm{\bar{g}_8}_{C^{1,\alpha}(\Omega^0)}
+\norm{\bar{g}_6}_{C^{2,\alpha}(\Omega^0)}\Big)
\end{eqnarray}
with the aid of \eqref{314}.

\subsection{Solving the Candidate Velocity on $\Omega^0$}
Now we need to solve $\hat{u}$ on $\Omega^0.$ To this end, by
\eqref{263} and \eqref{268}, we reformulate this problem as
\begin{eqnarray}\label{318}
&&d(\hat{\bar{u}}^1|_{\Omega^0})
=\chi(U,U^-,\psi):=-\frac{d\bar{g}_0(U,U^-,\psi,D\psi)}{\mu_0}\in
C^{2,\alpha}_2(\mathbf{S}^2), \\
&&d^*(\hat{\bar{u}}^1|_{\Omega^0})=\frac{\mu_5}{r^1-r_b}\p_{{0}}\hat{p}|_{\Omega^0}
+\mu_6(\hat{\psi}-r_b)+\bar{g}_5(U,U^-,\psi,DU,D\psi)\in
C^{2,\alpha}(\mathbf{S}^2),\label{319}
\end{eqnarray}
where $d$ and $d^*$ are respectively the exterior differential and
codifferential operator of forms on $\mathbf{S}^2$. By integrating
\eqref{309} on $\mathbf{S}^2$, we see that the integration of the
right-hand side of \eqref{319} on $\mathbf{S}^2$ also vanishes.
Then, by Lemma \ref{lema3}, there exists a unique solution
$\hat{\bar{u}}^1|_{\Omega^0}\in C_1^{3,\alpha}(\mathbf{S}^2)$ with
the estimate:
\begin{eqnarray}
\norm{\hat{\bar{u}}^1|_{\Omega^0}}_{C^{3,\alpha}_1(\mathbf{S}^2)}\le
C_2\Big(\norm{\bar{g}_0}_{C^{3,\alpha}_1(\Omega^0)}
+\norm{(\bar{g}_5, \bar{g}_6, \bar{g}_7)}_{C^{2,\alpha}(\Omega^0)}
+\norm{\bar{f}_{9}}_{C^{1,\alpha}(\Omega)}
+\norm{\bar{g}_8}_{C^{1,\alpha}(\Omega^0)}\Big).
\end{eqnarray}
Therefore, combining this with \eqref{264}, the velocity on the
candidate free boundary is obtained.

\subsection{Solving the Candidate Entropy, Density, and Vorticity in
$\Omega$} We note that the entropy can be solved according to
\eqref{107} and \eqref{267} by the following Cauchy problem of the
linear transport equations:
\begin{eqnarray}
&D_u\widehat{A(S)}=0 &\text{in}\ \ \Omega,\label{en}\\
&\widehat{A(S)}=\mu_4(\hat{\psi}-r_b)+\bar{g}_4(U,U^-,\psi,D\psi)
&\text{on}\ \ \Omega^0.\label{en1}
\end{eqnarray}
Note that $A(S)_b^+(x^0)=A(S)_b^+(r_b)$ is a constant. By the theory
of ordinary differential equations, since $u$ is close to
$(u^0)_b^+\p_0$ in $C^{3,\alpha}(\Omega)$, the trajectories of $u$
still fill $\Omega$. We also have the estimate:
\begin{eqnarray*}
\|\widehat{A(S)}\|_{C^{3,\alpha}(\Omega)}\le
C_2\Big(\|\hat{\psi}-r_b\|_{C^{3,\alpha}(\mathbf{S}^2)}
+\norm{\bar{g}_4}_{C^{3,\alpha}(\Omega^0)}\Big).
\end{eqnarray*}
Hence, we may solve the candidate density $\hat{\rho}+\rho_b^+$ by
the state function $p=A(S)\rho^\gamma$. Then
\begin{eqnarray*}
\norm{\hat{\rho}}_{C^{3,\alpha}(\Omega)}\le
C_2\Big(\norm{\hat{p}}_{C^{3,\alpha}(\Omega)}
+\norm{\widehat{A(S)}}_{C^{3,\alpha}(\Omega)}\Big).
\end{eqnarray*}

For $\widehat{d\bar{u}}=d\bar{u}=d\bar{u}-d\bar{u}_b^+$, by
subtracting the background solution from \eqref{183}, we formulate a
linear transport equation:
\begin{equation}\label{326}
\mathcal{L}_{u}\widehat{d\bar{u}}=-d(\frac{1}{\rho_b^+})\wedge
d\hat{p}+\frac{d\hat{\rho}\wedge d
p_b^+}{(\rho_b^+)^2}+\bar{f}_{10}(U,DU).
\end{equation}
According to \eqref{261}, the initial value on $\Omega^0$ can be
taken as
\begin{eqnarray}\label{327}
\widehat{d\bar{u}}|_{\Omega^0}&=&d(\hat{\bar{u}}|_{\Omega^0})
+E_1(U,\psi,DU,D\psi,D^2\psi)|_{\Omega^0}\wedge
d\hat{\psi}\nonumber\\
&&+(r^1-r_b)\Big((r^1-r_b)r_bg_{\alpha\beta}\hat{u}^\beta|_{\Omega^0}
-\frac{\p_\alpha \hat{p}}{\rho_b^+(u^0)_b^+}\Big|_{\Omega^0}\Big)dy^0\wedge dy^\alpha\nonumber\\
&&+\bar{g}_9(U,\psi,DU,D\psi).
\end{eqnarray}
Note that $\hat{\bar{u}}|_{\Omega^0}$ has been obtained in \S 3.4:
\begin{eqnarray}\label{c}
\hat{\bar{u}}|_{\Omega^0}=\hat{\bar{u}}^1|_{\Omega^0}+(r^1-r_b)\mu_1
(\hat{\psi}+\hat{r}^p-r_b)d{y}^0 +\bar{g}_{10}(U,U^-,\psi,D\psi).
\end{eqnarray}
Therefore, we may solve  $\widehat{d\bar{u}}$ from
\eqref{326}--\eqref{327} and obtain the estimate (see Lemma
\ref{lema5b}):
\begin{eqnarray*}
\|\widehat{d\bar{u}}\|_{C^{2,\alpha}_2(\Omega)}&\le&
C_2\Big(\norm{\bar{f}_{10}}_{C^{2,\alpha}_2(\Omega)}
+\norm{\bar{g}_{9}}_{C^{2,\alpha}_2(\Omega^0)}
+\norm{(\hat{\rho},\hat{p})}_{C^{3,\alpha}(\Omega)}\nonumber\\
&&\qquad+\|\hat{\psi}-r_b\|_{C^{3,\alpha}(\mathbf{S}^2)}+
\norm{\hat{\bar{u}}^1|_{\Omega^0}}_{C^{3,\alpha}_1(\mathbf{S}^2)}
+\norm{\bar{g}_{10}}_{C^{3,\alpha}(\Omega^0)}\Big).
\end{eqnarray*}

\subsection{Solving the Lower Regular Candidate Velocity
in $\Omega$ and on $\Omega^1$}

From \eqref{108}, by subtracting the background solution, we
formulate a transport equation of the velocity $\hat{\bar{u}}$ in
$\Omega$ (to distinguish the lower regular velocity obtained here
from the candidate velocity in the next subsection, we write $u$ as
$u_l$):
\begin{eqnarray}\label{a}
\lie_{u_b^+}\hat{\bar{u}}_l=-\frac{d\hat{p}}{\rho_b^+}
+\frac{\hat{\rho}dp_b^+}{(\rho_b^+)^2}+\bar{f}_{11}(U,\psi,DU,D\psi).
\end{eqnarray}
Here we used  that $\lie_{\hat{u}_l}\bar{u}_b^+-d\langle
\bar{u}_b^+,\hat{\bar{u}}_l\rangle=0$. With the Cauchy data
$\hat{\bar{u}}_l|_{\Omega^0}$ as the right-hand side of  \eqref{c},
we may uniquely solve $\hat{\bar{u}}_l$, particularly its
restriction on $\Omega^1$, i.e., $\hat{\bar{u}}_l|_{\Omega^1}$. We
also have an estimate:
\begin{eqnarray}\label{b}
\norm{\hat{\bar{u}}_l}_{C^{2,\alpha}_1(\Omega)}&\le& C_2
\Big(\norm{\bar{f}_{11}}_{C_1^{2,\alpha}(\Omega)}
+\norm{(\hat{p},\hat{\rho})}_{C^{3,\alpha}(\Omega)}
+\norm{\hat{\bar{u}}^1|_{\Omega^0}}_{C^{3,\alpha}_1(\mathbf{S}^2)}\nonumber\\
&& \qquad +\|\hat{\psi}-r_b\|_{C^{3,\alpha}(\mathbf{S}^2)}
+\norm{\bar{g}_{10}}_{C_1^{3,\alpha}(\Omega^0)}\Big).
\end{eqnarray}

\subsection{Solving the Candidate Velocity in $\Omega$}

Note that $\hat{\bar{u}}_l$ obtained in the above step is only in
$C_1^{2,\alpha}$; it is not our desired candidate velocity. In fact,
we will solve the velocity $\hat{\bar{u}}$ by the following elliptic
equation motivated by \eqref{186}:
\begin{eqnarray}\label{329}
\Delta \hat{\bar{u}}&=&d\langle \hat{\bar{u}}_l, \frac{d(\ln
p_b^+)}{\gamma}\rangle+d^*(\widehat{d\bar{u}})+d\big(\frac{D_{u_b^+}\hat{p}}{\gamma
p_b^+}-\frac{D_{u_b^+}p_b^+}{\gamma (p_b^+)^2}\hat{p}\big)
+\bar{f}_{12}(U,U^-,\psi,DU,D^2U).
\end{eqnarray}
We impose the Dirichlet condition \eqref{c} on $\Omega^0$ and the
Neumann condition on $\Omega^1$ according to \eqref{a} by
\begin{eqnarray}\label{187}
D_{u_b^+}\hat{\bar{u}}=D_{u_b^+}\hat{\bar{u}}_l-\lie_{u_b^+}\hat{\bar{u}}_l+\big(-\frac{d\hat{p}}{\rho_b^+}
+\frac{dp_b^+}{(\rho_b^+)^2}\hat{\rho}
+\bar{f}_{11}(U,\psi,DU,D\psi)\big)\big|_{\Omega^1}.
\end{eqnarray}
Note that, in local spherical coordinates,
\begin{eqnarray*}
D_{u_b^+}\hat{\bar{u}}_l-\lie_{u_b^+}\hat{\bar{u}}_l
=-(\hat{\bar{u}}_l)_0\p_0((u^0)_b^+)dx^0-\frac{1}{x^0}((u^0)_b^+)
((\hat{\bar{u}}_l)_\alpha dx^\alpha),
\end{eqnarray*}
so it does not contain the derivatives of $\hat{\bar{u}}_l$. Therefore,
we may solve $\hat{u}$ (i.e., $\hat{\bar{u}}$) in $\Omega$ by Lemma
\ref{lema4} to obtain
\begin{eqnarray*}
&&\norm{\hat{\bar{u}}}_{C_1^{3,\alpha}(\Omega)}\le
C_2\Big(\norm{\hat{\bar{u}}_l}_{C^{2,\alpha}_1(\Omega)}
  +\norm{d\hat{\bar{u}}}_{C^{2,\alpha}_2(\Omega)}
  +\norm{(\hat{p}, \hat{\rho})}_{C^{3,\alpha}(\Omega)}
  +\norm{\bar{f}_{12}}_{C_1^{1,\alpha}(\Omega)}
  +\norm{\bar{f}_{11}}_{C^{2,\alpha}(\Omega^1)}
\Big). \nonumber\\
&&
\end{eqnarray*}

\subsection{Well-Definedness of the
Mapping $\mathcal{T}: \mathcal{K}_\sigma\times O_\delta\to \mathcal{K}_\sigma\times O_\delta$}
First we notice that, for any $\psi\in \mathcal{K}_\sigma,\ U\in U_b^+ + O_\delta$,
and $U^-$ satisfying \eqref{015}, it is clear from the definition
that a higher order term $f$ satisfies
\begin{eqnarray}
\norm{f}\le C_2(\varepsilon+\delta^2+\sigma^2).
\end{eqnarray}
Then, combining this with the estimates in \S 3.2--\S 3.7 yields
\begin{eqnarray}
\|\hat{U}\|_{C^{3,\alpha}(\Omega)}+\|\hat{\psi}-r_b\|_{C^{4,\alpha}
(\mathbf{S}^2)}\le C_2(\delta^2+\sigma^2+\varepsilon).
\end{eqnarray}
Now we choose $C_0=2C_2$ and $\varepsilon_0\le \frac{1}{8C_2^2}$.
Then, for $\delta=\sigma=C_0\varepsilon$, the above estimate shows
that $\hat{U}\in O_\delta$ and $\hat{\psi}\in \mathcal{K}_\sigma.$

\subsection{Contraction of the Mapping}
Now, for $i=1,2$, choose arbitrarily  $\psi^{(i)}\in \mathcal{K}_\sigma$ and
$\check{U}^{(i)}\in O_{\delta}$, and let $\hat{U}^{(i)}$ and
$\hat{\psi}^{(i)}$ be obtained by the above process correspondingly.
Then
\begin{eqnarray}
\|\hat{U}^{(1)}-\hat{U}^{(2)}\|_{C^{2,\alpha}(\Omega)}\le
C_3\varepsilon
\Big(\|\check{U}^{(1)}-\check{U}^{(2)}\|_{C^{2,\alpha}(\Omega)}
+\|\psi^{(1)}-\psi^{(2)}\|_{C^{3,\alpha}(\mathbf{S}^2)}\Big),\label{333}\\
\|\hat{\psi}^{(1)}-\hat{\psi}^{(2)}\|_{C^{3,\alpha}(\mathbf{S}^2)}\le
C_3\varepsilon
\Big(\|\check{U}^{(1)}-\check{U}^{(2)}\|_{C^{2,\alpha}(\Omega)}
+\|\psi^{(1)}-\psi^{(2)}\|_{C^{3,\alpha}(\mathbf{S}^2)}\Big).\label{334}
\end{eqnarray}
This can be achieved by employing the equations of
$\hat{U}^{(1)}-\hat{U}^{(2)}$ and the estimates of higher order terms as
sketched below.

First, for any higher order term $f$, when
$\delta=\sigma=C_0\varepsilon$, $\psi^{(i)}\in \mathcal{K}_\sigma$,
and $\check{U}^{(i)}\in O_{\delta}$, we have
\begin{eqnarray*}
&&\|f(\psi^{(1)},U^{(1)},{U^-}^{(1)},\cdots)-f(\psi^{(2)},U^{(2)},{U^-}^{(2)},\cdots)\|_*\nonumber\\
&&\le
C_2\varepsilon\Big(\|\psi^{(1)}-\psi^{(2)}\|_{C^{3,\alpha}(\mathbf{S}^2)}
+\|\check{U}^{(1)}-\check{U}^{(2)}\|_{C^{2,\alpha}(\Omega)}\Big),
\end{eqnarray*}
where  $U^{(i)}:=\check{U}^{(i)}+U_b^+$,
${U^-}^{(i)}={\psi^{(i)}}^*(U^-)$, and $\|\cdot\|_*$ is the
corresponding norm for  $f$ when $\psi\in C^{3,\alpha}$ and $U\in
C^{2,\alpha}$. As two examples, we have

(a). By the mean value theorem and \eqref{015},
\begin{eqnarray}
&&\|(\psi^{(1)})^*(U^--U_b^-)
 -(\psi^{(2)})^*(U^--U_b^-)\|_{C^{2,\alpha}(\mathbf{S}^2)}\nonumber\\
&&\ \ \le C \|D(U^--U_b^-)\|_{C^{2,\alpha}(\mathcal{M})}\|\psi^{(1)}
-\psi^{(2)}\|_{C^{2,\alpha}(\mathbf{S}^2)}\le
C\varepsilon\|\psi^{(1)} -\psi^{(2)}\|_{C^{3,\alpha}(\mathbf{S}^2)};
\end{eqnarray}

(b). For
$\bar{f}_9^{(i)}=\bar{f}_9(U^{(i)},{U^-}^{(i)},\psi^{(i)},
DU^{(i)},D^2U^{(i)})$,
\begin{eqnarray}
\norm{\bar{f}_9^{(1)}-\bar{f}_9^{(2)}}_{C^\alpha(\Omega)} \le
C\varepsilon\Big(\|\psi^{(1)}-\psi^{(2)}\|_{C^{3,\alpha}(\mathbf{S}^2)}
+\|\check{U}^{(1)}-\check{U}^{(2)}\|_{C^{2,\alpha}(\Omega)}\Big)
\end{eqnarray}
by the definition of higher order terms.

Next, consider the equations of $\hat{U}^{(1)}-\hat{U}^{(2)}$. The
right-hand sides of the elliptic equations of
$\hat{p}^{(1)}-\hat{p}^{(2)}$, such as
$\bar{f}_9^{(1)}-\bar{f}_9^{(2)}$ (cf. \eqref{175}) and
$\bar{g}_8^{(1)}-\bar{g}_8^{(2)}$ (cf.\eqref{309}), are in
${C^\alpha}$. Therefore, we can  obtain a $C^{2,\alpha}$--estimate of
$\hat{U}^{(1)}-\hat{U}^{(2)}$, rather than a $C^{3,\alpha}$--estimate.
The reason is that the loss of derivative  occurs in solving
the transport equations. For example, from \eqref{en} and
\eqref{en1}, we have
\begin{eqnarray}
&&D_{u^{(1)}}(\widehat{A(S)}^{(1)}-\widehat{A(S)}^{(2)})
=-D_{u^{(1)}-u^{(2)}}\widehat{A(S)}^{(2)} \qquad\qquad \text{in}\ \ \Omega,\label{en2}\\
&&\widehat{A(S)}^{(1)}-\widehat{A(S)}^{(2)}\nonumber\\
&&=\mu_4(\hat{\psi}^{(1)}-\hat{\psi}^{(2)})
+\big(\bar{g}_4(U^{(1)},{U^-}^{(1)},\psi^{(1)},D\psi^{(1)})
-\bar{g}_4(U^{(2)},{U^-}^{(2)},\psi^{(2)},D\psi^{(2)})\big)
\quad\text{on} \ \Omega^0.\qquad \quad\label{en3}
\end{eqnarray}
Note that the right-hand side of \eqref{en2} is only in
$C^{2,\alpha}(\Omega)$.

We omit the details of deriving estimates \eqref{333}--\eqref{334},
since
the process is similar to those in \S 3.2--\S 3.7 except the two points explained above.

Then the mapping $\mathcal{T}$ has a unique fixed point if
$\varepsilon_0$ is small enough by a simple generalized Banach fixed
point theorem. In particular, the uniqueness of the fixed point implies
the transonic shock solution of \eqref{101}--\eqref{103} satisfying
the requirements in Theorem \ref{thm101} is unique, as claimed there. We
also note that, although $\mathcal{T}$ always has a fixed point as we
proved, it is not clear yet whether this  fixed point is a
solution to \eqref{101}--\eqref{103}; therefore, in order to obtain the
existence result as assumed in Theorem \ref{thm101}, it requires
further work, which is out of scope of this paper.

\bigskip
\begin{appendix}
\section{Some Notations and Facts of Differential Geometry}

In this appendix, we present some notations in differential geometry and
some basic facts used above for self-containedness.

\subsection{The Metric of $\mathcal{M}$ in Local Spherical Coordinates}
In local spherical coordinates, for $r^0\le x^0\le r^1, 0\le
x^1<\pi, -\pi\le x^2<\pi,$ the standard Euclidean metric of
$\mathcal{M}$ can be written  as
\[
G=G_{ij}dx^i\otimes dx^j=dx^0\otimes dx^0+(x^0)^2dx^1\otimes
dx^1+(x^0\, \sin x^1)^2dx^2\otimes dx^2.
\]
Hence, $\sqrt{G}:=\sqrt{\det (G_{ij})}=(x^0)^2\sin x^1$. For the
Christoffel symbols, since $\Gamma_{jk}^i=\Gamma^i_{kj}$,
only the following are nonzero:
\begin{eqnarray*}
&&\Gamma_{11}^0=-x^0, \quad \Gamma^0_{22}=-x^0(\sin x^1)^2, \quad
\Gamma^1_{01}=\Gamma^1_{10}=\frac{1}{x^0},\\
&&\Gamma^1_{22}=-\sin x^1\cos x^1, \quad
\Gamma^2_{02}=\Gamma^2_{20}=\frac{1}{x^0}, \quad
\Gamma^2_{12}=\Gamma^2_{21}=\cot x^1.
\end{eqnarray*}
We also use $(G^{ij})$ to denote the inverse of the matrix
$(G_{ij})$, and $|u|^2=G(u,u)=G_{ij}u^iu^j$.

In local spherical coordinates, we write the standard metric of
$\mathbf{S}^2$ as
$$
g=g_{\alpha\beta}dx^\alpha\otimes
dx^\beta:=dx^1\otimes dx^1+(\sin x^1)^2dx^2\otimes dx^2.
$$
Therefore, we have
$$
\sqrt{g}:=\sqrt{\det (g_{ij})}=\sin x^1=\frac{\sqrt{G}}{(x^0)^2}.
$$
The nonzero Christoffel symbols are
\begin{eqnarray}
\gamma_{22}^1=-\sin x^1\cos x^1,\qquad
\gamma_{12}^2=\gamma^2_{21}=\cot x^1.
\end{eqnarray}

\subsection{Some Lemmas} The following results are used in the text.

\begin{lemma}\label{lema3}
There exists a unique $\omega\in A^1(\mathbf{S}^2)$ that solves
\begin{eqnarray}\label{a4}
d\omega=\chi,\qquad d^*\omega=\psi,
\end{eqnarray}
if $\chi\in A^2(\mathbf{S}^2)$ and $\psi\in A^0(\mathbf{S}^2)$
satisfy  $\int_{\mathbf{S}^2}\chi=0$ and
$\int_{\mathbf{S}^2}\psi\,\vol=0$.
\end{lemma}

\begin{proof}
Since the first Betti number $b_1$ of $\mathbf{S}^2$ is $0$ (i.e.,
$b_1=0$), by the Hodge theorem, we can uniquely solve $\omega$ via
\begin{eqnarray}
\Delta \omega=d^*\chi+d\psi\in A^1(\mathbf{S}^2).
\end{eqnarray}
It suffices to show that \eqref{a4} holds.

First, since $\Delta d=d\Delta$, we have
\begin{eqnarray}
\Delta(d\omega-\chi)=d\Delta\omega-(dd^*\chi+d^*d\chi)=0.
\end{eqnarray}
In addition, $\int_{\mathbf{S}^2} (d\omega-\chi)=0$ by the Stokes
theorem and the assumption. Since the second Betti number $b_2$ of
$\mathbf{S}^2$ is $1$,  by the Hodge theorem, the space
$\mathcal{H}^2$ of harmonic $2$-forms on $\mathbf{S}^2$ is
one-dimensional. Note that $\vol\in\mathcal{H}^2$ because $d^*=-*d*$
and $*\vol=1$. Therefore, $d\omega-\chi=0$.

Similarly, we have
\begin{eqnarray}
\Delta(d^*\omega-\psi)&=&d^*\Delta\omega-(dd^*\psi+d^*d\psi)=0
\end{eqnarray}
due to $\Delta d^*=d^*\Delta$, and
$$
\int_{\mathbf{S}^2}(d^*\omega-\psi)\,\vol=0,
$$
by the divergence theorem and the assumption. Note that the zero-th
Betti number $b_0$ of $\mathbf{S}^2$ is $1$;  according to the Hodge
theorem, the space $\mathcal{H}^0$ of harmonic functions on
$\mathbf{S}^2$ is one-dimensional. One easily sees that $1\in
\mathcal{H}^0$. Therefore, $d^*\omega-\psi=0$ as desired.
\end{proof}

\begin{lemma}\label{lema4}
Let the two disjoint components of the boundary
$\partial\mathcal{M}$ be $\mathcal{M}^0$ and $\mathcal{M}^1$,
$k=2,3$, and $\alpha\in(0,1)$. Assume that $\beta\in
C_1^{k-2,\alpha}(\mathcal{M})$, $\omega_0$ is a $C^{k,\alpha}$
$1$-form, and $\omega_1$ is a $C^{k-1,\alpha}$ $1$-form in
$\mathcal{M}$. Then there exists a unique $C^{k,\alpha}$ $1$-form
$\omega$ that solves the following problem:
\begin{eqnarray}
&&\Delta \omega=\beta \qquad \text{in}\ \ \mathcal{M},\label{A7}\\
&&\omega|_{\mathcal{M}^0}=\omega_0|_{\mathcal{M}^0},\label{a8}\\
&&D_{\p_0}
\omega|_{\mathcal{M}^1}=\omega_1|_{\mathcal{M}^1}.\label{a9}
\end{eqnarray}
Moreover,
\begin{eqnarray}\label{a10}
\norm{\omega}_{C_1^{k,\alpha}(\mathcal{M})}\le
C\Big(\norm{\beta}_{C_1^{k-2,\alpha}(\mathcal{M})}
+\norm{\omega_0}_{C_1^{k,\alpha}(\mathcal{M})}
+\norm{\omega_1}_{C_1^{k-1,\alpha}(\mathcal{M})}\Big).
\end{eqnarray}
\end{lemma}

\begin{proof}
For $\mathcal{M}\subset\mathbb{R}^3$, we choose the spherical
coordinates:
$$
\tilde{y}^0=(r^1-r_b)y^0+r_b, \quad \tilde{y}^\alpha=y^\alpha,
\qquad \alpha=1,2.
$$
Then we use the standard global Descartes coordinates:
$$
(z^0,z^1,z^2) \qquad \text{with}\,\,\,
 \tilde{y}^0=\sqrt{(z^0)^2+(z^1)^2+(z^2)^2}.
$$
Let $\omega=\omega_idz^i$. By the Weizenb\"{o}ck formula,
\begin{equation}\label{A.11}
 \Delta \omega=-\sum_{i=0}^2(\p_{z^i}^2 \omega_j) dz^j
\end{equation}
holds globally (cf. \cite{Fr}). Therefore, \eqref{A7}--\eqref{a9}
represent three decoupled boundary value problems of the Poisson
equations. The uniqueness, existence, and estimates of the solution
are then clear.
\end{proof}

The following result follows from  Lemma 4.6 in \cite{ChF1}. It
particularly implies that the norm of a smooth function  in a
manifold $\Omega_2$ is equivalent to the norm of its pull back in
another manifold $\Omega_1$ which is homeomorphic to $\Omega_2$.

\begin{lemma}\label{lemaa}
      Let $\Omega_1$ and $\Omega_2$ be two open sets in $\mathbb{R}^n$ and
      $u\in C^{k,\alpha}({\Omega}_2)$.
Let $k$ be a positive integer and $\alpha\in (0,1)$.
      Let $\Phi: \Omega_1\rightarrow\Omega_2$
      satisfy $\Phi\in C^{k,\alpha}({\Omega}_1;\Omega_2)$. Then $u\circ\Phi\in
      C^{k,\alpha}({\Omega}_1)$ and satisfies
      \begin{equation}\label{4310}
            \|u\circ\Phi\|_{C^{k,\alpha}({\Omega}_1)}
                 \le C\|u\|_{C^{k,\alpha}({\Omega}_2)},
      \end{equation}
     where $C=C(n, \|\Phi\|_{C^{k,\alpha}({\Omega}_1;\Omega_2)}).$
\end{lemma}

\begin{lemma} [\cite{BC}] \label{lema5a}
The eigenvalues of the Hodge Laplacian
on
$\mathbf{S}^2$ are $\lambda_n=n(n+1), n=0, 1, 2, \cdots$, and there
are $2n+1$ linear independent eigenfunctions
$\{u_{n,m}\}_{m=1,\cdots, 2n+1}$ corresponding to $\lambda_n$, so
that the eigenfunctions $\{u_{n,m}\}_{n\in
\mathbb{N}\cup\{0\},m=1,\cdots,2n+1}$ are smooth and form a complete
unit orthogonal basis  of $L^2(\mathbf{S}^2)$.
\end{lemma}

\subsection{On the Transport Equations Involving Lie Derivatives in
Manifolds} In Section 3.5, we need solve the differential forms from
the Cauchy problems of the transport equations involving Lie derivatives.
Here we present the basic theorem with a proof.

Let $M$ be an $n$-dimensional closed $C^{\infty}$ differentiable
manifold, $\mathcal{M}=[0,T]\times {M}$, $X$ a $C^{k+1}$-vector
field in $\mathcal{M}$ which is
transverse to $\Gamma^t=\{t\}\times M$ for $t\in[0,T]$, and $f$ a
$C^k$--function in $\mathcal{M}$ ($k$ is a nonnegative integer).
Without loss of generality, we
assume that $X$ points to the interior of $\mathcal{M}$ when
restricted on $\Gamma^0$. We wish to solve a $r$-form $\omega$
($r\ge1$) in $\mathcal{M}$ which satisfies the following problem:
\begin{eqnarray}
&\lie_X\omega+f \omega=\theta & \text{in}\quad \mathcal{M},\label{1}\\
&\omega=\omega^0 & \text{on}\quad \Gamma^0.\label{2}
\end{eqnarray}
Here $\lie$ is the Lie derivative in $\mathcal{M}$,  $\theta$ is a
given $C^k$ $r$-form in $\mathcal{M}$, and $\omega^0$ is a given
point-wise defined $r$-form of class $C^k$ on $\Gamma^0.$

We have the following existence and uniqueness results:

\begin{lemma}\label{lema5b}
Under the above assumptions, there is a unique $r$-form $\omega$
in $\mathcal{M}$ that solves \eqref{1}--\eqref{2}. In addition, there
holds
\begin{eqnarray}\label{estimate}
\norm{\omega}_{C^{k}(\mathcal{M})}\le
C\big(\norm{\theta}_{C^{k}(\mathcal{M})}+\norm{\omega^0}_{C^{k}
(\Gamma^0)}\big),
\end{eqnarray}
with a positive constant $C$ depending only on
$\norm{f}_{C^{k}(\mathcal{M})}$ and
$\norm{X}_{C^{k+1}(\mathcal{M})}.$
\end{lemma}

For the proof, we first get familiar what \eqref{1} stands for in a
local coordinate chart.

Let $E$ be a local coordinate chart of $M$. Then
$\tilde{E}=[0,T]\times E$ is a coordinate chart of $\mathcal{M}$. In
$\tilde{E}$, problem \eqref{1}--\eqref{2} is an initial value problem
of the transport equations. To see this, for simplicity, suppose that
$X=X^0\p_0+X^\alpha\p_\alpha$,  with $ X^\alpha=0$ for
$\alpha=1,\cdots,n$ in $\mathcal{M}$. Since $x^0$ is a global
coordinate, $X^0=dx^0(X)$ is a $C^{k}$--function defined in
$\mathcal{M}$. By our assumption, $X^0$ is positive and bounded away
from zero since $\mathcal{M}$ is compact.

Suppose that
\begin{eqnarray}
&&\omega=\omega_{i_1\cdots i_r}dx^{i_1}\wedge\cdots\wedge
dx^{i_r}=r!\sum_{0\le i_1<\cdots< i_r\le n}\omega_{i_1\cdots
i_r}dx^{i_1}\wedge\cdots\wedge dx^{i_r},\\
&&\theta=\theta_{i_1\cdots i_r}dx^{i_1}\wedge\cdots\wedge
dx^{i_r}=r!\sum_{0\le i_1<\cdots< i_r\le n}\theta_{i_1\cdots
i_r}dx^{i_1}\wedge\cdots\wedge dx^{i_r}.
\end{eqnarray}
Since, for the differential forms $\alpha$ and $\beta$, there hold
$$
\lie_X(\alpha\wedge\beta)=(\lie_X\alpha)\wedge\beta+\alpha
\wedge(\lie_X\beta),
$$
$d\lie_X\alpha=\lie_X(d\alpha)$, and
$\lie_Xf=Xf$ (cf. \cite{Fr}),
we obtain
\begin{eqnarray}
\lie_X\omega&=&r!\sum_{0<i_2<\cdots<i_r\le
n}\Big(X^0\p_0\omega_{0i_2\cdots i_r}+ \omega_{0 i_2\cdots
i_r}\p_0X^0\Big)dx^0\wedge
dx^{i_2}\wedge\cdots\wedge dx^{i_r}\nonumber\\
&&+r!\sum_{0<i_1<\cdots<i_r\le n}X^0\p_0\omega_{i_1\cdots
i_r}dx^{i_1}\wedge\cdots\wedge dx^{i_r}\nonumber\\
&&+r!\sum_{0<i_2<\cdots<i_r\le n}\omega_{0i_2\cdots i_r}\p_\alpha
X^0dx^\alpha\wedge dx^{i_2}\wedge\cdots\wedge dx^{i_r}\nonumber\\
&=&r!\sum_{0<i_2<\cdots<i_r\le n}\Big(X^0\p_0\omega_{0i_2\cdots
i_r}+ \omega_{0 i_2\cdots i_r}\p_0X^0\Big)dx^0\wedge
dx^{i_2}\wedge\cdots\wedge dx^{i_r}\nonumber\\
&&+r!\sum_{0<i_1<\cdots<i_r\le n}\Big(X^0\p_0\omega_{i_1\cdots
i_r}+\frac{1}{(r-1)!}\nonumber\\
&&\sum_{\sigma\in \mathcal{P}(r)}\big(\sign(\sigma)\
\omega_{0i_{\sigma(2)}\cdots i_{\sigma(r)}}\p_{i_{\sigma(1)}}
X^0\big)\Big) dx^{i_1}\wedge dx^{i_2}\wedge\cdots\wedge dx^{i_r}.
\end{eqnarray}
Here $\mathcal{P}(r)$ is the permutation group of $\{1,\cdots, r\}$,
and $\sign(\sigma)$ is the sign of a permutation $\sigma$.

Hence, by dividing $X^0$ from both sides of equation \eqref{1}, we
have
\begin{equation}
\p_0\omega_{0i_2\cdots
i_r}+\big(\frac{1}{X^0}\p_0X^0+\frac{f}{X^0}\big)\omega_{0i_2\cdots
i_r}=\frac{1}{X^0}\theta_{0i_2\cdots i_r},\quad
 0<i_2<\cdots<i_r\le
n;\label{7}
\end{equation}
and
\begin{eqnarray}
&&\p_0\omega_{i_1\cdots i_r}+\frac{f}{X^0}\omega_{i_1\cdots
i_r}\nonumber \\
&& =\frac{1}{X^0}\theta_{i_1\cdots
i_r}-\frac{1}{(r-1)!X^0}\sum_{\sigma\in\mathcal{P}(r)}\big(\sign(\sigma)
\p_{i_{\sigma(1)}}X^0\omega_{0i_{\sigma(2)}
\cdots i_{\sigma(r)}}\big),\quad  0<i_1<\cdots<i_r\le n.\qquad\quad\label{8}
\end{eqnarray}
These are linear transport equations  in the $x^0$-direction.

We can first solve \eqref{7} in $\tilde{E}$ by using the initial data
$\omega_{0i_2\cdots i_r}|_{\Gamma^0}=(\omega^0)_{0i_2\cdots i_r}$
and then substitute $\omega_{0i_2\cdots i_r}$ in the right-hand side
of \eqref{8} to solve $\omega_{i_1\cdots i_r}$ in $\tilde{E}$ with
initial data $\omega_{i_1i_2\cdots
i_r}|_{\Gamma^0}=(\omega^0)_{i_1i_2\cdots i_r}$. By antisymmetry of
the lower indices, we obtain all the coefficients $\omega_{i_1\cdots
i_r}$. Since the quantities in \eqref{1}--\eqref{2} are defined
globally, $\omega=\omega_{i_1\cdots i_r}dx^{i_1}\wedge\cdots
dx^{i_r}$ we solved is also well defined in $\mathcal{M}$.
Estimate \eqref{estimate} is  obvious from these initial value
problems of ordinary differential equations.

\smallskip

For the general case that $X\ne X^0\p_0$,  equation \eqref{1} would
be a first-order hyperbolic system with $\binom{n+1}{r}$ unknowns.
In addition, we can not use the rather simple coordinate charts like
$\tilde{E}$, since the
characteristic curves (i.e., integral curves of $X$) may escape
$\tilde{E}$ at  some  $t<T$  by solving the initial value problem.

We note that the Lie derivative behaves well under the homeomorphisms of
differentiable manifolds.

\begin{proposition}
Let $\Phi$ be a homeomorphism of $\mathcal{M}$. Then
\begin{eqnarray}
(\Phi^{-1})^*\lie_X\omega=\lie_{\Phi_*X}((\Phi^{-1})^*\omega).
\end{eqnarray}
\end{proposition}

This can be shown by using the Cartan formula $\lie_X\omega=di_X\omega+i_Xd\omega$  and
the formulae  $$\Phi^*d\omega=d\Phi^*\omega,\
(\Phi^{-1})^*(i_X\omega)=i_{\Phi_*X}(\Phi^{-1})^*\omega.$$ (cf.
\cite{Fr}), where $i_X\omega$ is the interior product of $\omega$ and
$X$.

In addition, by Lemma \ref{lemaa}, the $C^{k}$--norm of a
differential form in $\mathcal{M}$ is equivalent under $C^{k}$--homeomorphism
of the manifold. Thus, to prove Lemma \ref{lema5b} for the general case, it suffices
to straighten the vector field $X$ to the form
$X^0\p_0$ globally by a suitable homeomorphism of $\mathcal{M}$.

\begin{proposition}\label{lema7}
For given $C^{k+1}$--vector field $X$ in $\mathcal{M}$ which is
transverse to $\{x^0\}\times M$ for every $x^0\in[0,T]$,  there is a
$C^{k+1}$--homeomorphism $\Phi$: $\mathcal{M}\rightarrow\mathcal{M}$
such that, for any fixed $x^0\in[0,T]$, it is also a homeomorphism
of $\{x^0\}\times M$, and $\Phi_*X=(\Phi^{-1})^*(dx^0(X))\p_0.$
\end{proposition}

\begin{proof}
{\it Step 1.}\  Let $\tilde{E}$ be a coordinate chart of
$\mathcal{M}$ as introduced above and  $X=X^i\p_i$ with
$X^0>0$. Since $X^0$ is a nonzero function in $\mathcal{M}$, we may
define a $C^{k+1}$--vector field $\tilde{X}$ in $\mathcal{M}$ by
\begin{eqnarray}
\tilde{X}=\tilde{X}^i\p_i=\p_0+\frac{X^\alpha}{X^0}\p_\alpha.
\end{eqnarray}

{\it Step 2.}\ Now $\tilde{X}$ generates a flow $\phi_t$ in
$\mathcal{M}$ by theory of ordinary differential equations since
$\mathcal{M}$ is compact: For any $P\in M$, $\gamma(t)=\phi_t(P),\,
t\in[0,T]$, is a curve in $\mathcal{M}$ with initial value
$\gamma(0)=(0,P)\in \Gamma^0$. We then define $\Phi:
\mathcal{M}\rightarrow\mathcal{M}$ by
\begin{eqnarray}
\Phi(\phi_t(P))=(t, P), \qquad P\in M.
\end{eqnarray}
Note that $\phi_t(P)\in \Gamma^t$ since $\tilde{X}^0=1.$ We now show
that $\Phi$ is a homeomorphism.

\medskip
{\it Step 3.}\ For any $Q\in\mathcal{M}$, it is easy to see that
there is uniquely a pair $(t,P)$ with  $P\in M$ and $t\in[0,T]$ such
that $\phi_t(P)=Q$ by solving backward the integral curve of
$\tilde{X}$ through $Q$. So $\Phi$ is defined for all the points in
$\mathcal{M}$. In addition, $\Phi$ is obviously surjective and
injective by the uniqueness and existence results of the initial value
problem of ordinary differential equations. By continuous dependence
on the initial data $(0,P)$ and $t$, we see that $\Phi$ and $\Phi^{-1}$ are
also continuous. If $\tilde{X}\in C^{k+1}$, then $\Phi$ and
$\Phi^{-1}$ are $C^{k+1}$--mappings by $C^{k+1}$--dependence of
solutions of ordinary differential equations on $t$ and initial
data. This proves that  $\Phi$ is a $C^{k+1}$--homeomorphism.

\medskip
{\it Step 4.}\ Now,  by definition of push-forward mapping (tangent
mapping) of vectors, we have
\begin{eqnarray}
\Phi_*(X)&=&\Phi_*(X^0\tilde{X})=(\Phi^{-1})^*(X^0)\Phi_*(\tilde{X})\nonumber\\
&=&(\Phi^{-1})^*(X^0)\frac{d}{dt}\Phi(\gamma(t))\nonumber\\
&=&(\Phi^{-1})^*(X^0)\p_0.
\end{eqnarray}
This completes the proof.
\end{proof}

\end{appendix}

\bigskip
{\bf Acknowledgments.}
The research of
Gui-Qiang G. Chen was supported in part by the National Science
Foundation under Grants
DMS-0935967 and DMS-0807551, the UK EPSRC Science and Innovation
Award to the Oxford Centre for Nonlinear PDE (EP/E035027/1),
the NSFC under a joint project Grant 10728101, and
the Royal Society--Wolfson Research Merit Award (UK).
Hairong Yuan's research is
supported in part by Shanghai Chenguang Program (09CG20), National
Natural Science Foundation of China under Grants NSFC-10901052,
NSFC-10871071, Shanghai Leading Academic Discipline Project (No.
B407), and a Fundamental Research Funds for the Central
Universities.

%%%----------------------------------------------------------------------------

\end{document}